\documentclass[12pt,showkeys]{amsart}
\def\makeheadbox{{%
\hbox to0pt{\vbox{\baselineskip=10dd\hrule\hbox
to\hsize{\vrule\kern3pt\vbox{\kern3pt
\hbox{\bfseries Draft for discussion }
\hbox{Date of this version: 13.04.21}
\kern3pt}\hfil\kern3pt\vrule}\hrule}%
\hss}}}

\pdfoutput=1

\usepackage{color}
\usepackage{graphicx}
\usepackage{subfigure}
\usepackage{amssymb}
\usepackage{amsmath}
\usepackage{accents}
\usepackage{multirow}
\usepackage{enumerate}
\usepackage{epstopdf}
\usepackage{cite,stmaryrd,txfonts}
\usepackage{tikz}
\usepackage{pgf}
\usetikzlibrary{snakes}
\usetikzlibrary{angles,calc}
\usetikzlibrary{shapes,arrows,automata}
 \usepackage{float}

\usepackage{epic}
\usepackage{empheq}
\usepackage{cases}
\usepackage{diagbox}
\def\cequiv{\raisebox{-1.5mm}{$\;\stackrel{\raisebox{-3.9mm}{=}}{{\sim}}\;$}}

\newtheorem{theorem}{Theorem}[section]
\newtheorem{remark}[theorem]{Remark}
\newtheorem{proposition}[theorem]{Proposition}
\newtheorem{lemma}[theorem]{Lemma}

\newtheorem{definition}[theorem]{Definition}

\newtheorem{assumption}{Assumption}

\newcounter{mnote}
\let\oldmarginpar\marginpar
\renewcommand\marginpar[1]{\-\oldmarginpar[\raggedleft\footnotesize #1]
  {\raggedright\footnotesize #1}}

\usepackage{footnote}
\usepackage{booktabs}

\usepackage{rotating}

\numberwithin{equation}{section}

\usepackage{cite}

\usepackage[colorlinks,linkcolor=black,
anchorcolor=black,
citecolor=black]{hyperref}
\usepackage[shortlabels]{enumitem}
\setlist[enumerate]{nosep}
\usepackage{color, soul}
\soulregister\cite7
\usepackage{setspace}

\def\uu{\undertilde{u}}
\def\uw{\undertilde{w}}
\def\uv{\undertilde{v}}
\def\uV{\undertilde{V}}

\def\uZ{\undertilde{Z}}

\def\uH{\undertilde{H}}

\def\uP{\undertilde{P}}

\def\uphi{\undertilde{\varphi}}
\def\upsi{\undertilde{\psi}}

\def\uf{\undertilde{f}}

\def\uy{\undertilde{y}}

\def\rot{{\rm rot}}
\def\curl{{\rm curl}}
\def\dv{{\rm div}}
\begin{document}
\title[A lowest-degree strictly conservative finite element scheme]{A lowest-degree strictly conservative finite element scheme for incompressible Stokes	problem on general triangulations}

\author{Wenjia Liu}

\author{Shuo Zhang}
\address{LSEC, Institute of Computational Mathematics and Scientific/Engineering Computing, Academy of Mathematics and System Sciences, Chinese Academy of Sciences, Beijing 100190; University of Chinese Academy of Sciences, Beijing, 100049; People's Republic of China}
\email{\{wjliu,szhang\}@lsec.cc.ac.cn}

\thanks{The research is supported by NSFC (11871465) and CAS (XDB 41000000).}



\begin{abstract}
In this paper, we propose a finite element pair for incompressible Stokes problem. The pair uses a slightly enriched piecewise linear polynomial space for velocity and piecewise constant space for pressure, and is illustrated to be a lowest-degree conservative stable pair for the Stokes problem on general triangulations.
\end{abstract}

\maketitle


%
%
%
\section{Introduction}
\label{sec;introduction}

For the Stokes problem, if a stable finite element pair can inherit the mass conservation, the approximation of the velocity can be independent of the pressure and the method does not suffer from the locking effect with respect to large Reynolds' numbers (cf., e.g.,~\cite{Brezzi;Fortin1991}). Over the past decade, the conservative schemes have been recognized more clearly as \emph{pressure robustness} and widely studied and surveyed in, e.g.,~\cite{Nicolas;Alexander;Philipp2019,Guzman;NeilanGeneral2014,Neilan2017,Schroeder;Lube2018}.  
This conservation is also connected to other key features like ``viscosity-independent" ~\cite{Uchiumi2019}, ``gradient-robustness" ~\cite{Linke;Merdon2020}, etc for numerical schemes. The importance of conservative schemes is also significant in, e.g., the nonlinear mechanics~\cite{Auricchio;Veiga;Lovadina;Reali2010,Auricchio2013} and the magnetohydrodynamics~\cite{Hu;Ma;Xu2017,Hiptmair;Li;Mao;Zheng2018,Hu;Xu2019}. Wide interests have been drawn to conservative schemes.  
~\\

Various conservative finite element pairs have been designed for the Stokes problem. Conforming examples include conforming elements designed for special meshes, such as $\undertilde{P}_{k}-P_{k-1}$ triangular elements for $k\geqslant 4$ on singular-vertex-free meshes~\cite{Scott;Vogelius2009} and for smaller $k$ constructed on composite grids~\cite{Arnold;Qin1992,Scott;Vogelius2009,Qin;Zhang2007,Zhang2008,XuZhang2010} and the pairs given in \cite{FALK;NEILAN;2013,Guzman;NeilanGeneral2014} which work for general triangulations and with extra smoothness requirement. An alternative method is to use $H({\rm div})$-conforming but $H^1$-nonconforming space for the velocity. A systematic approach is to add bubble-like functions onto $H(\dv)$ finite element spaces for the tangential weak continuity for the velocity. Examples along this line can be found in, e.g., \cite{Madal.K;Tai.X;Winther.R2002,Guzman.J;Neilan.M2012,Tai;Winther2006} and \cite{Xie;Xu;Xue2008}. Generally, to construct a conservative pair that works on general triangulations without special structures, cubic and higher-degree polynomials are used for the velocity. 
~\\

Recently, a new $\undertilde{P}_2-P_1$ finite element pair is proposed on general triangulations; for the velocity field, it uses piecewise quadratic $H(\dv)$ functions with enhanced tangential continuity, and for the pressure, it uses discontinuous piecewise linear functions. The pair is stable and immediately strictly conservative on general triangulations, and is of the lowest degree ever known. Meanwhile, as is pointed out in \cite{Zeng.H;Zhang.C;Zhang.S2021}, this $\undertilde{P}_2-P_1$ pair can be viewed as a smoothened reduction from the famous Brezzi-Douglas-Marini pair, and this idea can be carried on for other $H(\dv)$ pairs so that the degree of finite element pairs may be reduced further. 
~\\

In this paper, we study how low can the degree of polynomials be to construct a stable conservative pair that works on general triangulations. We begin with the reduction of the 2nd order Brezzi-Douglas-Fortin-Marini element pair to construct an auxiliary finite element pair $\uV^{\rm sBDFM}_{h0}-\mathbb{P}^{1}_{h0}$, and then a further reduction of the $\uV^{\rm sBDFM}_{h0}-\mathbb{P}^{1}_{h0}$ pair leads to a $\uV^{\rm el}_{h0}-\mathbb{P}^0_{h0}$ pair. The finally proposed pair, as the centerpiece of this paper, uses a slightly enriched linear polynomial space for the velocity and piecewise constant for the pressure, and is stable and conservative. A further reduction of this pair leads to a $\undertilde{P}_1-P_0$ pair which is constructed naturally but not stable on general triangulations, and this way we find the newly designed $\uV^{\rm el}_{h0}-\mathbb{P}^0_{h0}$ pair is one of lowest degree. We note that this $\uV^{\rm el}_{h0}-\mathbb{P}^0_{h0}$ pair is of the type ``nonconforming spline'' and can not be represented by Ciarlet's triple. However, the velocity space does admit a set of basis functions with quite tight local supports, which are clearly stated in Section \ref{sec:elpair}.
~\\ 

The main technical ingredients of the paper are two folded. One is to figure out the basis functions, the supports of which are quite different from existing finite elements. Another is to prove the stability (inf-sup condition). We mainly utilize a two-step argument. For the auxiliary pair $\uV^{\rm sBDFM}_{h0}-\mathbb{P}^{1}_{h0}$, we mainly utilize Stenberg's macroelement argument by following the procedures of \cite{Zeng.H;Zhang.C;Zhang.S2021}; then the stability of the pair $\uV^{\rm el}_{h0}-\mathbb{P}^0_{h0}$, which is a sub-pair of $\uV^{\rm sBDFM}_{h0}-\mathbb{P}^{1}_{h0}$, is proved just by inheriting the stability of the $\uV^{\rm sBDFM}_{h0}-\mathbb{P}^{1}_{h0}$. This ``reduce and inherit'' procedure can be found in, e.g., \cite{Zhang.S2020ima,Zhang.S2021+scm} where some low degree optimal schemes are designed for other problems. It can be a natural idea to generalize all technical ingredients here to other applications. 
~\\

The rest of the paper is organized as follows. In the remaining of this section, we present some standard notations. Some preliminaries on finite elements are surveyed in Section \ref{sec:pre}.  In Section~\ref{sec:sbdfm}, a smoothened BDFM element and an auxiliary stable conservative pair $\uV^{\rm sBDFM}_{h0}-\mathbb{P}^{1}_{h0}$, are established. In Section \ref{sec:bh}, a low-degree continuous nonconforming scheme for the biharmonic equation is presented.  In Section \ref{sec:elpair}, a low-degree stable conservative pair $\uV^{\rm el}_{h0}-\mathbb{P}^0_{h0}$ is constructed, while it is verified numerically in Appendix~\ref{sec:app} that a further reduction of the degree leads to an unstable pair. In Section \ref{sec:num}, some numerical experiments are reported to illustrate the effect of the schemes given in the present paper. Finally, in Section \ref{sec:conc}, some concluding remarks are given.

\subsection{Notations}

In what follows, we use $\Omega$ to denote a simply connected polygonal domain. We use $\nabla$, $\curl$, $\dv$, $\rot$, and $\nabla^2$ to denote the gradient operator, curl operator, divergence operator, rot operator, and Hessian operator, respectively. As usual, we use $H^2(\Omega)$, $H^2_0(\Omega)$, $H^1(\Omega)$, $H^1_0(\Omega)$, $H(\rot,\Omega)$, $H_0(\rot,\Omega)$, and $L^2(\Omega)$ to denote certain Sobolev spaces, and specifically, denote $\displaystyle L^2_0(\Omega):=\{w\in L^2(\Omega):\int_\Omega w dx=0\}$, $\undertilde{H}{}^1_0(\Omega):=(H^1_0(\Omega))^2$. Furthermore, we denote vector-valued quantities by $``\undertilde{~}"$, while $\uv{}^1$ and $\uv{}^2$ denote the two components of the function $\uv$. We use $(\cdot,\cdot)$ to represent $L^2$ inner product, and $\langle\cdot,\cdot\rangle$ to denote the duality between a space and its dual. Without ambiguity, we use the same notation $\langle\cdot,\cdot\rangle$ for different dualities, and it can occasionally be treated as $L^2$ inner product for certain functions. We use the subscript $``\cdot_h"$ to denote the dependence on triangulation. In particular, an operator with the subscript $``\cdot_h"$ indicates that the operation is performed cell by cell. Finally, $\cequiv$ denotes equality up to a constant. The hidden constants depend on the domain, and when triangulation is involved, they also depend on the shape regularity of the triangulation, but they do not depend on $h$ or any other mesh parameter.

The two complexes below are well known.
\begin{equation}\label{eq:derhamc}
	\{0\} \xrightarrow{\rm inc} H^1_0(\Omega) \xrightarrow{{\curl}} H_0(\dv,\Omega) \xrightarrow{{\dv}} L^2_0(\Omega) \xrightarrow{\int_\Omega\cdot} \{0\},
\end{equation}

\begin{equation}\label{eq:stokesc}
	\{0\} \xrightarrow{\rm inc} H^2_0(\Omega) \xrightarrow{{\curl}} \uH^1_0(\Omega) \xrightarrow{{\dv}} L^2_0(\Omega) \xrightarrow{\int_\Omega\cdot} \{0\}.
\end{equation}
We refer to, e.g., \cite{Arnold.D2018book,Arnold.D;Falk.R;Winther.R2006} for related discussion on more complexes and finite elements. 

The fundamental incompressible Stokes problem reads:
\begin{equation}\label{eq:Stokes eq}
	\left\{
	\begin{split}
		-\varepsilon^{2} \Delta\,\uu  + \nabla\, p & = \uf,  \quad \mbox{in} \ \Omega, \\
		\dv\, \uu & =  0, \quad \mbox{in} \ \Omega,  \\
		\uu  & = \undertilde{0}, \quad \mbox{on} \ \partial\Omega.
	\end{split}
	\right.
\end{equation}
Here $\uu$ stands for the velocity field and $p$ for the pressure field of the incompressible flow, and $\varepsilon^2$ stands for the inverse of the Reynold's number, which can be very small. Its variational formulation is to find $(\uu,p)\in \uH^1_0(\Omega)\times L^2_0(\Omega)$, such that 
\begin{equation}\label{eq:Stokes vf}
	\left\{
	\begin{aligned}
		&\varepsilon^{2}\big(\nabla\,\uu, \nabla\,\uv\big)  -( \dv\,\uv, p) && = ( \uf,\uv ), & \forall\, \uv\in \uH^1_0(\Omega), \\
		&(\dv\,\uu, q )&& = 0, & \forall\, q\in L^2_0(\Omega).
	\end{aligned}
	\right.
\end{equation}

\section{Preliminaries}
\label{sec:pre}

\subsection{Triangulations}

Let $\mathcal{T}_h$ be a shape-regular triangular subdivision of $\Omega$ with mesh size $h$, such that $\overline\Omega=\cup_{T\in\mathcal{T}_h}\overline T$. Denote by $\mathcal{T}_h$, $\mathcal{T}^i_h$, $\mathcal{E}_h$, $\mathcal{E}_h^i$, $\mathcal{E}_h^b$, $\mathcal{X}_h$, $\mathcal{X}_h^i$ and $\mathcal{X}_h^b$ the set of cells, cells with three interior edges, edges, interior edges, boundary edges, vertices, interior vertices and boundary vertices, respectively. For any edge $e\in\mathcal{E}_h$, denote by $\mathbf{n}_e$ and $\mathbf{t}_e$ the unit normal and tangential vectors of $e$, respectively. The subscript ${\cdot}_e$ can be dropped when there is no ambiguity.

Denote
$$
\mathcal{X}_h^{b,+1}:=\{a\in\mathcal{X}_h^i,\ a\ \mbox{is\ connected\ to}\ \mathcal{X}_h^b\ by\ e\in\mathcal{E}_h^i\},\ \mbox{and}\ \ \mathcal{X}_h^{i,-1}:=\mathcal{X}_h^i\setminus\mathcal{X}_h^{b,+1};
$$
further, denote with $\mathcal{X}^{i,-(k-1)}_h\neq\emptyset$,
$$
\mathcal{X}_h^{b,+k}:=\{a\in\mathcal{X}_h^{i,-(k-1)},\ a\ \mbox{is\ connected\ to}\ \mathcal{X}_h^{b,+(k-1)}\ by\ e\in\mathcal{E}_h^i\}, \ \mbox{and}\ \ \mathcal{X}_h^{i,-k}:=\mathcal{X}_h^{i,-(k-1)}\setminus\mathcal{X}_h^{b,+k}.
$$
The smallest $k$ such that $\mathcal{X}_h^{i,-(k-1)}=\mathcal{X}_h^{b,+k}$ is called the number of layers of the triangulation.

\tikzset{
	dot/.style={
		circle, fill=black, inner sep=.6pt, outer sep=0pt
	},
	dot label/.style={
		circle, inner sep=0pt, outer sep=1pt
	},
	pics/right angle/.append style={
		/tikz/draw, /tikz/angle radius=5pt
	}
}
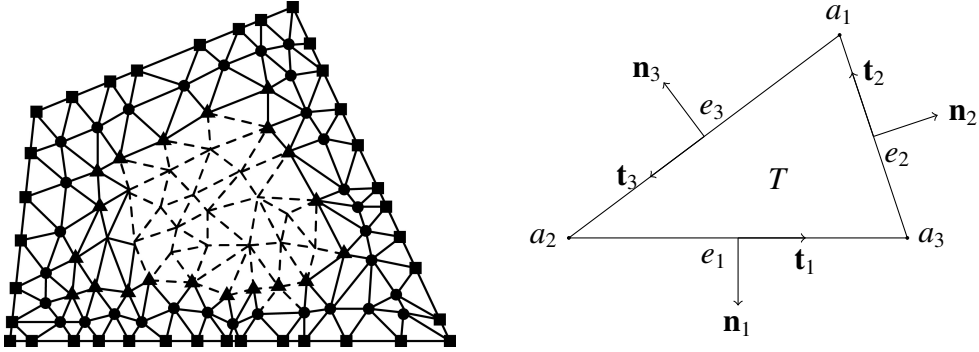
\begin{figure}[hp]
	\centering
	\includegraphics[width=6cm]{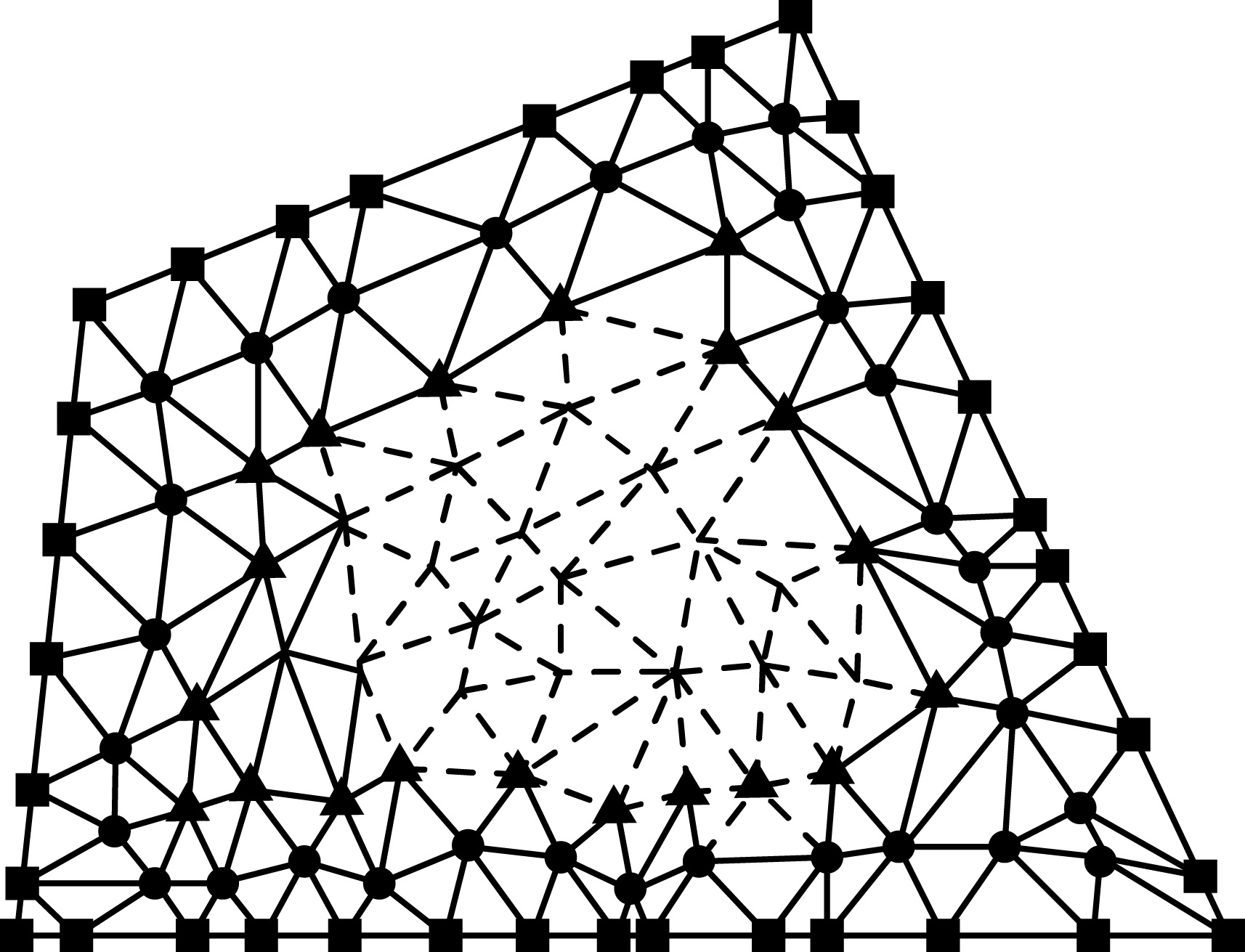}
	\hspace{0.25in}
	\label{fig:refT}
	\begin{tikzpicture}[scale=0.9]
		\path coordinate (B) at (0,0)
		coordinate (A) at (4,3)
		coordinate (C) at (5,0);
		
		\path coordinate (A1) at ($ (B)!.5!(C) $)
		coordinate (B1) at ($ (A)!.5!(C) $)
		coordinate (C1) at ($ (A)!.5!(B) $)
		coordinate (A2) at ($ (A1)!1cm!-90:(C) $)
		coordinate (B2) at ($ (B1)!1cm!-90:(A) $)
		coordinate (C2) at ($ (C1)!1cm!-90:(B) $)
		coordinate (A3) at ($ (A1)!1cm!0:(C) $)
		coordinate (B3) at ($ (B1)!1cm!0:(A) $)
		coordinate (C3) at ($ (C1)!1cm!0:(B) $)
		coordinate (A4) at ($ (A)!0.66cm!41:(C1) $)
		coordinate (B4) at ($ (B)!0.42cm!24:(A1) $)
		coordinate (C4) at ($ (C)!0.2cm!-18:(B1) $)
		coordinate (O) at ($ (A)!.6!(A1) $);
		
		\draw[line width=.4pt]  (A) -- (B) -- (C) -- cycle;
		\draw [->]    (A1) -- (A2);
		\draw [->]    (B1) -- (B2);
		\draw [->]    (C1) -- (C2);
		
		\draw [->]    (A1) -- (A3);
		\draw [->]    (B1) -- (B3);
		\draw [->]    (C1) -- (C3);
		
		\node [below] at (O) {$T$};
		
		\node[below] at (A3) {$\mathbf{t}_1$};
		\node[below] at (A2) {$\mathbf{n}_1$};
		
		\node[right] at (B3) {$\mathbf{t}_2$};
		\node[right] at (B2) {$\mathbf{n}_2$};
		
		\node[left] at (C3) {$\mathbf{t}_3$};
		\node[left] at ($(B)!0.58!21:(A)$) {$\mathbf{n}_3$};

		\draw[fill] (A) circle [radius=0.02];
		\draw[fill] (B) circle [radius=0.02];
		\draw[fill] (C) circle [radius=0.02];
		
		\node[above] at (A) {$a_1$};
		\node[left] at (B) {$a_2$};
		\node[right] at (C) {$a_3$};
		
		\node[below left] at (A1) {$e_1$};
		\node[below right] at (B1) {$e_2$};
		\node[above] at ($(B)!0.53!(A)$) {$e_3$};
		%
	\end{tikzpicture}
	\caption{Left: vertex layers, where the $\blacksquare$'s denote boundary vertices, the $\bullet$'s denote vertices of $\mathcal{X}_h^{b,+1}$, the $\blacktriangle$'s denote vertices of $\mathcal{X}_h^{b,+2}$, and so forth. Right: a reference triangle.}\label{vertex layers and ref triangle}
\end{figure}

On a triangle $T$, locally we use $\{a_1,a_2,a_3\}$ to denote its three vertices and $\{e_1,e_2,e_3\}$ to denote three edges with unit outward normal vectors $\{\mathbf{n}_1, \mathbf{n}_2, \mathbf{n}_3\}$ and unit tangential vectors $\{\mathbf{t}_1, \mathbf{t}_2, \mathbf{t}_3\}$ such that $\mathbf{n}_i \times \mathbf{t}_i >0,i\in \{1,2,3\}$; see Figure \ref{vertex layers and ref triangle}(right) for an illustration. In addition $\{\lambda_1, \lambda_2, \lambda_3\}$ are the barycentric coordinates with respect to the three corners of $T$. Also denote the lengths of edges by $\{d_1,d_2,d_3\}$, and the area of $T$ by $S_T$ and drop the subscript when no ambiguity is brought in.

\begin{figure}[H]
	\centering
	\begin{tikzpicture}[scale=0.6]
		\path coordinate (O) at (0,0)
		coordinate (A) at (-90:3)
		coordinate (B) at (-160:3)
		coordinate (C) at (-20:3)
		coordinate (D) at (150:3.5)
		coordinate (E) at (30:3.5)
		coordinate (F) at (108:3.5)
		coordinate (G) at (72:3.5)
		coordinate (V0) at (5:0.2);
		
		\node[below,font=\small] at (V0) {$O$};
		
		\draw[line width=.4pt]  (O) -- (A) ;
		\draw[line width=.4pt]  (O) -- (B) ;
		\draw[line width=.4pt]  (O) -- (C) ;
		\draw[line width=.4pt]  (O) -- (D) -- (B) -- (A) -- (C) -- (E) -- cycle ;
		\draw[line width=.4pt]  (O) -- (F) -- (G) -- cycle ;
		\draw[dashed] (D) -- (F) ;
		\draw[dashed] (E) -- (G) ;
		
		\path coordinate (B1) at ($(A)!.5!(B)$)
		coordinate (B2) at ($(A)!.4!(C)$)
		coordinate (B3) at ($(B)!.5!(D)$)
		coordinate (B4) at ($(C)!.5!(E)$)
		coordinate (B5) at ($(F)!.5!(G)$);
		
		\path coordinate (C1) at (-93:1.6)
		coordinate (C2) at (-153:1.4)
		coordinate (C3) at (-23:1.6)
		coordinate (C4) at (150:1.8)
		coordinate (C5) at (30:2.2)
		coordinate (C6) at (106:2.1)
		coordinate (C7) at (74:2.1);

		\path coordinate (A1) at (-130:2)
		coordinate (A2) at (-40:2)
		coordinate (A3) at (-182:2)
		coordinate (A4) at (0:2)
		coordinate (A5) at (90:2.5)
		coordinate (A6) at (134:1.8)
		coordinate (A7) at (62:1.8);
		
		\node[above] at (A6) {$ \cdots $};
		\node[right] at (A7) {$ \cdots $};
		
		\path coordinate (D1) at (-92:2.5)
		coordinate (D2) at (-88:2.45)
		coordinate (D3) at (-159:2.4)
		coordinate (D4) at (-174:2.55)
		coordinate (D5) at (-10:2.6)
		coordinate (D6) at (-30:2.2)
		coordinate (D7) at (151:2.9)
		coordinate (D8) at (28:2.7)
		coordinate (D9) at (88.5:3.1)
		coordinate (D10) at (109:3.25);
	\end{tikzpicture}
	\hspace{0.5in}
	\begin{tikzpicture}[scale=0.6]
		\path coordinate (A) at (0,0)
		coordinate (B) at (60:3)
		coordinate (C) at (120:3)
		coordinate (E) at (180:3)
		coordinate (F) at (0:3)
		coordinate (D) at ($ (C)!-1!(E) $);
		
		\draw[line width=.4pt]  (A) -- (B) -- (C) -- cycle;
		\draw[line width=.4pt]  (A) -- (F) -- (B) -- (D) -- (C) -- (E)-- cycle;
		
		\path coordinate (A1) at ($ (B)!.5!(C) $)
		coordinate (B1) at ($ (A)!.5!(C) $)
		coordinate (C1) at ($ (A)!.5!(B) $)
		coordinate (A2) at ($ (A1)!.4!(C) $)
		coordinate (B2) at ($ (B1)!.4!(A) $)
		coordinate (C2) at ($ (C1)!.4!(B) $)
		coordinate (A3) at ($ (A1)!0.6cm!90:(B) $)
		coordinate (B3) at ($ (B1)!0.6cm!90:(C) $)
		coordinate (C3) at ($ (C1)!0.6cm!90:(A) $)
		coordinate (O) at ($ (A1)!0.66cm!-90:(B) $);
		
		\path   coordinate (C4) at ($ (C)!0.25cm!-10:(A) $)
		coordinate (B4) at ($ (B)!0.5cm!-10:(C) $)
		coordinate (A4) at ($ (A)!0.7cm!-10:(B) $)
		coordinate (D1) at ($ (D)!.35!(A1) $)
		coordinate (E1) at ($ (E)!.35!(B1) $)
		coordinate (F1) at ($ (F)!.35!(C1) $);
		
		\node [below,font=\small] at (O) {$T_0$};
	\end{tikzpicture}
	\caption{Illustration of an interior vertex patch(left) and an interior cell patch(right).}
	\label{fig:basissupp simple}
\end{figure}
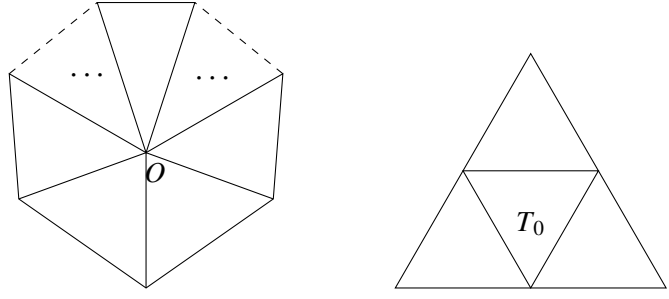
Next, we figure out two types of patches.
\begin{description}
	\item[{\bf interior vertex patch:}] for an interior vertex $O$, the cells that connects to $O$ form a (closed) interior vertex patch, denoted by $P_O$; see Figure \ref{fig:basissupp simple}(left) for an illustration;
	\item[{\bf interior cell patch:}] for an interior cell $T_0$, three neighbored cells and $T_0$ form an interior cell patch, denoted by $P_{T_0}$; see Figure \ref{fig:basissupp simple}(right) for an illustration.
\end{description} 

The number of interior vertex patches is $\# \mathcal{X}_h^i$ and the number of interior cell patches is $\# \mathcal{T}^i_h ( = 2 \# \mathcal{X}_h^i - 2 ) $. 

In the sequel, we impose a mild assumption on the grid.
\begin{assumption}\label{assumption grid}
	Every boundary vertex is connected to at least one interior vertex. 
\end{assumption}
This assumption assures every cell is covered by at least one interior vertex patch. 

\subsection{Polynomial spaces on a triangle}

For a triangle $T$, we use $P_k(T)$ to denote the set of polynomials on $K$ of degrees not higher than $k$. In a similar manner, $P_k(e)$ is defined on an edge $e$. We define $\uP{}_k(T)=(P_k(T))^2$ and similarly is $\uP{}_k(e)$ defined. 

Following \cite{Madal.K;Tai.X;Winther.R2002}, we introduce the shape function space:
$$
\uP^{\rm MTW}(T):=\{\uv\in\uP_3(T):\uv\cdot\mathbf{n}|_{e_i}\in P_1(e_i),\ i=1:3,\ \dv\, \uv\ \mbox{is\ a\ constant\ on}\ T\}.
$$
It can be verified (cf. \cite{Guzman.J;Neilan.M2012}) that
$$
\uP^{\rm MTW}(T)=\uP{}_1(T)\oplus{\rm span}\{\curl(\lambda_i^2\lambda_j\lambda_k)\}_{\{i,j,k\}=\{1,2,3\}}.
$$

Following \cite{Guzman.J;Neilan.M2012}, we introduce the shape functions space
$$
\uP^{\rm GN-1}(T)=\uP_1(T)\oplus\{\curl(\lambda_i^2\lambda_j^2\lambda_k)\}_{\{i,j,k\}=\{1,2,3\}}.
$$

We further denote
$$
\displaystyle
\uP{}^{2-}(T):=\uP{}_1(T)\oplus{\rm span}\{\lambda_i\lambda_j\mathbf{t}_k\}_{\{i,j,k\}=\{1,2,3\}},\ \ \ \mbox{and}\ \ \ 
\uP{}^{1+}(T):=\uP{}_1(T)\oplus{\rm span}\{\curl (\lambda_1\lambda_2\lambda_3)\}.
$$
It can be verified that $\uP{}^{1+}(T)\subset\uP{}^{2-}(T)$, and
$$
\uP{}^{2-}(T)=\{\uv\in\uP{}_2(T):\uv\cdot\mathbf{n}|_{e_i}\in P_1(e_i),\ i=1:3\}, \ \ \ \mbox{and} \ \ \ \uP{}^{1+}(T)=\{\uv\in \uP{}^{2-}(T):\dv\,\uv\ \mbox{is\ a\ constant\ on}\ T\}.
$$

Further we denote
$$
P^{2+}(T):=P_2(T)\oplus{\rm span}\{\lambda_1\lambda_2\lambda_3\}.
$$

\begin{lemma}\label{lem:locales}
	The two exact sequences hold:
	\begin{equation}\label{eq:localesp2-}
		\mathbb{R}  \rightarrow P^{2+}(T) \xrightarrow{{\curl}}  \uP{}^{2-}(T) \xrightarrow{{\dv}} P_1(T),
	\end{equation}
	and
	\begin{equation}\label{eq:localesp1+}
		\mathbb{R}  \rightarrow P^{2+}(T) \xrightarrow{{\curl}}  \uP{}^{1+}(T) \xrightarrow{{\dv}} P_0(T).
	\end{equation}
\end{lemma}
\begin{proof}
	Noting that $\uP{}^{2-}(T)$ is exactly the local shape functions space of the quadratic Brezzi-Douglas-Fortin-Marini element, that $\dv\,\uP^{2-}(T)=P_1(T)$ is well known. Evidently $\curl\, P^{2+}(T)\subset \{\uv\in\uP^{2-}(T):\dv\,\uv=0\}$, and $\dim(\curl\, P^{2+}(T))=\dim(P^{2+}(T))-1=\dim(\uP^{2-}(T))-\dim(P_1(T))=\dim(\{\uv\in\uP^{2-}(T):\dv\,\uv=0\})$, thus $\curl\, P^{2+}(T)= \{\uv\in\uP^{2-}(T):\dv\,\uv=0\}$. The proof of \eqref{eq:localesp2-} is completed. Similarly, that $\dv\,\uP{}^{1+}(T) = P_0(T)$ follows by the definition of $\uP^{1+}(T)$, and \eqref{eq:localesp1+} can be proved the same way.
\end{proof}

Define for $i=1:3,\,\undertilde{w}_{T,e_i} := \curl(\lambda_j \lambda_k (3\lambda_i-1))$, $\undertilde{w}_{T,e_j,e_k} := \curl(\lambda_i^2)$ and $\uy_{T,e_j,e_k} := -\frac{2}{d_i} \lambda_i \mathbf{n}_i $. It holds trivially that $\dv\, \uw_{T,e_i} = 0$, $\dv\, \uw_{T,e_j,e_k} = 0$ and $\dv\, \uy_{T,e_j,e_k} = \frac{1}{S}$. 
It also indicates that $\undertilde{w}_{T,e_i}$ is a function with vanishing normal components and tangential integral on the edges $e_j,e_k$ and similar is $\undertilde{w}_{T,e_j,e_k}$ on the edge $e_i$. For instance, we refer to Figure \ref{T_ei and T_ejek} for an illustration of $\uw_{T,e_1}$ and $\uw_{T,e_2,e_3}$.
\begin{figure}[H]
	\centering
	\begin{tikzpicture}[scale=0.9]
		\path coordinate (O) at (0,0)
		coordinate (A) at (90:2)
		coordinate (B) at (-162:3)
		coordinate (C) at (-18:3);
		\path coordinate (D) at ($(B)!.5!(C)$)
		coordinate (E) at ($(A)!.45!(C)$)
		coordinate (F) at ($(A)!.43!(B)$);
		
		\draw[dashed]  (B) -- (A) -- (C) ;
		\draw[line width=.4pt] (B) -- (C) ;
		
		\node[above] at (A) {$ a_1 $};
		\node[left] at (B) {$ a_2 $};
		\node[right] at (C) {$ a_3 $};
		\node[below] at (D) {$ e_1 $};
		\node[right] at (E) {$ e_2 $};
		\node[left] at (F) {$ e_3 $};
	\end{tikzpicture}
	\hspace{0.25in}
	\begin{tikzpicture}[scale=0.9]
		\path coordinate (O) at (0,0)
		coordinate (A) at (90:2)
		coordinate (B) at (-162:3)
		coordinate (C) at (-18:3);
		
		\path coordinate (D) at ($(B)!.5!(C)$)
		coordinate (E) at ($(A)!.45!(C)$)
		coordinate (F) at ($(A)!.43!(B)$);
		
		\draw[line width=.4pt]  (B) -- (A) -- (C) ;
		\draw[dashed] (B) -- (C) ;
		
		\node[above] at (A) {$ a_1 $};
		\node[left] at (B) {$ a_2 $};
		\node[right] at (C) {$ a_3 $};
		\node[below] at (D) {$ e_1 $};
		\node[right] at (E) {$ e_2 $};
		\node[left] at (F) {$ e_3 $};
	\end{tikzpicture}
	\caption{Degrees of freedom vanish on dotted edges.}\label{T_ei and T_ejek}
\end{figure}
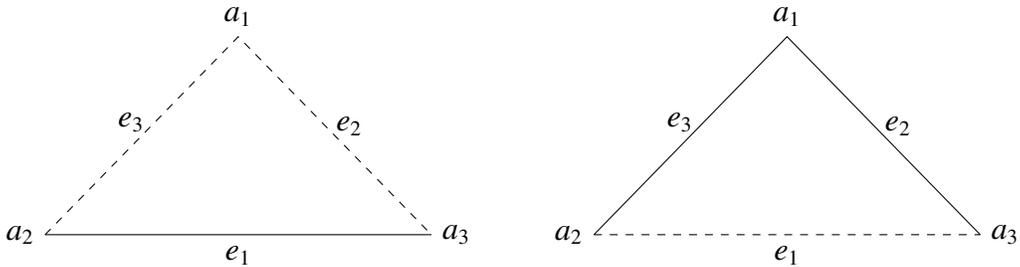
Then
\begin{multline}
	\uZ{}_T:=\{\uv\in\uP{}^{2-}(T):\dv\,\uv=0\}=\{\uv\in\uP{}^{1+}(T):\dv\,\uv=0\}
	\\
	={\rm span}\{\uw{}_{T,e_1},\uw{}_{T,e_2},\uw{}_{T,e_3},  \uw{}_{T,e_2,e_3},\uw{}_{T,e_3,e_1},\uw{}_{T,e_1,e_2} \}
\end{multline}
and
\begin{equation}\label{set P_tilde^1+}
	\uP^{1+}(T)={\rm span}\{\uw{}_{T,e_1},\uw{}_{T,e_2},\uw{}_{T,e_3},  \uw{}_{T,e_2,e_3},\uw{}_{T,e_3,e_1},\uw{}_{T,e_1,e_2}, \uy{}_{T,e_2,e_3},\uy{}_{T,e_3,e_1},\uy{}_{T,e_1,e_2}\}.
\end{equation}
Indeed, the functions of the set in (\ref{set P_tilde^1+}) are not linearly independent. Any one among $\{ \uy{}_{T,e_2,e_3},\uy{}_{T,e_3,e_1},\uy{}_{T,e_1,e_2} \}$ together with $\{\uw{}_{T,e_1},\uw{}_{T,e_2},\uw{}_{T,e_3},  \uw{}_{T,e_2,e_3},\uw{}_{T,e_3,e_1},\uw{}_{T,e_1,e_2}\}$ forms a set of independent basis of $\uP{}^{1+}(T)$. 

\subsection{Some known finite elements}

The Madal-Tai-Winther element (see \cite{Madal.K;Tai.X;Winther.R2002}) is defined by
\begin{enumerate}[(1)]
	\item $T$ is a triangle;
	\item $P_T=\uP{}^{MTW}(T)$;
	\item for any $\undertilde{v}\in (H^1(T))^2$, the nodal parameters on $T$, denoted by $D_T$, are\\
	$\{\fint_{e_i} \undertilde{v}\cdot \textbf{n}_{T,e_i} d\tau,
	\fint_{e_i} \undertilde{v}\cdot \textbf{n}_{T,e_i} (\lambda_j-\lambda_k)d\tau,
	\fint_{e_i} \undertilde{v}\cdot \textbf{t}_{T,e_i} d\tau\}_{i=1:3}$.
\end{enumerate}
Following \cite{Madal.K;Tai.X;Winther.R2002}, introduce
\begin{equation}
	\uV^{\rm MTW}_h:=\{\uv_h\in H(\dv,\Omega):\uv_h|_{T}\in \uP^{\rm MTW}(T),\ \int_{e}\uv\cdot\mathbf{t}\ \mbox{is\ continuous\ across\ interior\ edge}\ e\},
\end{equation}
and
\begin{equation}
	\uV^{\rm MTW}_{h0}:=\{\uv_h\in \uV^{\rm MTW}_h\cap H_0(\dv,\Omega):\ \int_{e}\uv\cdot\mathbf{t}=0\ \mbox{on\ boundary\ edge}\ e\}.
\end{equation}
The lowest-degree Guzman-Neilan element (see \cite{Guzman.J;Neilan.M2012}) is defined as
\begin{enumerate}[(1)]
	\item $T$ is a triangle;
	\item $P_T=\uP{}^{\rm GN-1}(T)$;
	\item for any $\undertilde{v}\in (H^1(T))^2$, the nodal parameters on $T$, denoted by $D_T$, are\\
	$\{\fint_{e_i} \undertilde{v}\cdot \textbf{n}_{T,e_i} d\tau,
	\fint_{e_i} \undertilde{v}\cdot \textbf{n}_{T,e_i} (\lambda_j-\lambda_k)d\tau,
	\fint_{e_i} \undertilde{v}\cdot \textbf{t}_{T,e_i} d\tau\}_{i=1:3}$.
\end{enumerate}
Following \cite{Guzman.J;Neilan.M2012}, introduce
\begin{equation}
	\uV^{\rm GN-1}_h:=\{\uv_h\in H(\dv,\Omega):\uv_h|_{T}\in P^{\rm GN-1}(T),\ \int_{e}\uv\cdot\mathbf{t}\ \mbox{is\ continuous\ across\ interior\ edge}\ e\},
\end{equation}
and
\begin{equation}
	\uV^{\rm GN-1}_{h0}:=\{\uv_h\in \uV^{\rm GN-1}_h\cap H_0(\dv,\Omega):\ \int_{e}\uv\cdot\mathbf{t}=0\ \mbox{on\ boundary\ edge}\ e\}.
\end{equation}
Following Zeng-Zhang-Zhang\cite{Zeng.H;Zhang.C;Zhang.S2021}, introduce
\begin{equation}
	\uV^{\rm ZZZ}_h:=\{\uv_h\in H(\dv,\Omega):\uv_h|_{T}\in \uP_2(T),\ \int_{e}\uv\cdot\mathbf{t}\ \mbox{is\ continuous\ across\ interior\ edge}\ e\},
\end{equation}
and
\begin{equation}
	\uV^{\rm ZZZ}_{h0}:=\{\uv_h\in \uV^{\rm ZZZ}_h\cap H_0(\dv,\Omega):\ \int_{e}\uv\cdot\mathbf{t}=0\ \mbox{on\ boundary\ edge}\ e\}.
\end{equation}
As revealed by \cite{Zeng.H;Zhang.C;Zhang.S2021}, the space can be viewed as a reduced Brezzi-Douglas-Marini element space with enhanced smoothness.

\subsection{Stenberg's macroelement technique for inf-sup condition (cf. \cite{Stenberg1990technique})}
A macroelement partition of $\mathcal{T}_h$, denoted by $\mathcal{M}_h$, is a set of macroelements satisfying that each triangle of $\mathcal{T}_h$ is covered by at least one macroelement in $\mathcal{M}_h$.
\begin{definition}
	Two macroelements $M_1$ and $M_2$ are said to be equivalent if there exists a continuous one-to-one mapping $G:M_1  \rightarrow M_2$, such that
	\begin{enumerate}[(a)]
		\item $G(M_1) = M_2$
		\item if $M_1=\bigcup_{i=1:m}^m T_i^1$, then $T_i^2=G(T_i^1)$ with $i=1:m$ are the cells of $M_2$.
		\item $G|_{T_i^1} = F_{T_i^2}\circ F_{T_i^1}^{-1}, i=1:m,$ where $F_{T_i^1}$ and $F_{T_i^2}$ are the mappings from a reference element $\hat{T}$ onto $T_i^1$ and $T_i^2$, respectively.
	\end{enumerate}
\end{definition}
A class of equivalent macroelements is a set of which any two macroelements are equivalent to each other.  Given a macroelement $M$, $\undertilde{V}_{h0,M}$, a subspace of $\uV_h$, consists of functions in $\uV_h$ that are equal to zero outside $M$; continuity constraints of $\uV_h$ enable corresponding nadal parameters of functions in $\undertilde{V}_{h0,M}$ to be zero on $\partial M$. 
Similarly, $Q_{h,M}$ is a subspace of $Q_h$ and it consists of functions that are equal to zero outside $M$. 
Denote
\begin{equation}\label{N_M}
	N_M:=\{q_h \in Q_{h,M}:\int_M div\ \undertilde{v}_h \ q_h \,dM=0, \forall\, \undertilde{v}_h \in \undertilde{V}_{h0,M}\}.
\end{equation}
Stenberg's macroelement technique can be summarized as the following proposition.
\begin{proposition}\label{macroelement technique}
	Suppose there exist a macroelement partitioning $\mathcal{M}_h$ with a fixed set of equivalence classes $\mathbb{E}_i$ of macroelements, $i=1,2,...,n$, a positive integer $N$ ($n$ and $N$ are independent of $h$), and an operator $\Pi:H_0^1(\Omega) \rightarrow \undertilde{V}_{h0}$, such that
	\begin{enumerate}[($C_1$)]
		\item for each $M\in \mathbb{E}_i,i=1,2,...,n$, the space $N_M$ defined in (\ref{N_M}) is one-dimensional, which consists of functions that are constant on M;
		\item each $M\in \mathcal{M}_h$ belongs to one of the classes $\mathbb{E}_i, i=1,2,...,n$;
		\item each $e \in \mathcal{E}_h^i$ is an interior edge of at least one and no more than $N$ macroelements;
		\item for any $\undertilde{w} \in \undertilde{H}_0^1(\Omega)$, it holds that\\
		$$\sum_{T\in \mathcal{T}_h} h_T^{-2} ||\undertilde{w} - \Pi \undertilde{w}||_{0,T}^2
		+ \sum_{e\in \mathcal{E}_h^i} h_e^{-1} ||\undertilde{w} - \Pi \undertilde{w}||_{0,e}^2
		\leqslant C ||\undertilde{w}||_{1,\Omega}^2 \quad and\quad ||\Pi \undertilde{w}||_{1,h} \leqslant C ||\undertilde{w}||_{1,\Omega}.$$
	\end{enumerate}
	Then the uniform inf-sup condition holds for the finite element pair.
\end{proposition}

\section{An auxiliary stable pair for the Stokes problem}
\label{sec:sbdfm}
\subsection{A smoothened Brezzi-Douglas-Fortin-Marini ({\rm sBDFM}) element}
We define sBDFM element by
\begin{enumerate}[(1)]
	\item $T$ is a triangle;
	\item $P_T=\uP{}^{2-}(T)$;
	\item for any $\undertilde{v}\in (H^1(T))^2$, the nodal parameters on $T$, denoted by $D_T$, are\\
	$\{\fint_{e_i} \undertilde{v}\cdot \textbf{n}_{T,e_i} d\tau,
	\fint_{e_i} \undertilde{v}\cdot \textbf{n}_{T,e_i} (\lambda_j-\lambda_k)d\tau,
	\fint_{e_i} \undertilde{v}\cdot \textbf{t}_{T,e_i} d\tau\}_{i=1:3}$.
\end{enumerate}
The above triple is $P_T-$unisolvent. We use $\undertilde{\varphi}_{\textbf{n}_{T,e_i},0}$, $\undertilde{\varphi}_{\textbf{n}_{T,e_i},1}$, and $\undertilde{\varphi}_{\textbf{t}_{T,e_i},0}$ to represent the corresponding nodal basis functions, and then
\begin{equation}\label{FE basis}
	\left\{
	\begin{aligned}
		& \undertilde{\varphi}_{\textbf{n}_{T,e_i},0} =
		\lambda_j (3 \lambda_j - 2) \frac{\textbf{t}_k}{(\textbf{n}_i,\textbf{t}_k)}
		+ \lambda_k (3 \lambda_k - 2) \frac{\textbf{t}_j}{(\textbf{n}_i,\textbf{t}_j)}
		+ 6 \lambda_j \lambda_k \textbf{n}_i;\\
		& \undertilde{\varphi}_{\textbf{n}_{T,e_i},1} =
		3\lambda_j (3 \lambda_j - 2) \frac{\textbf{t}_k}{(\textbf{n}_i,\textbf{t}_k)}
		- 3\lambda_k (3 \lambda_k - 2) \frac{\textbf{t}_j}{(\textbf{n}_i,\textbf{t}_j)};\\
		& \undertilde{\varphi}_{\textbf{t}_{T,e_i},0} =
		6 \lambda_j \lambda_k \textbf{t}_i.
	\end{aligned}
	\right.
\end{equation}

We use $\uV^{\rm sBDFM}_h$ and $\uV^{\rm sBDFM}_{h0}$ for the corresponding finite element spaces, where the subscript $\cdot_{h0}$ implies that the nodal parameters along boundary of the domain are all zero. Evidently, $\uV_h{}^{\rm sBDFM}$ is a {\bf s}moothened subspace of the famous Brezzi-Douglas-Fortin-Marini element space. Indeed $\undertilde{V}^{\rm sBDFM}_h \subset \undertilde{H}(div,\Omega)$ but $\undertilde{V}^{\rm sBDFM}_h \nsubset \undertilde{H}^1(\Omega)$, and similar is $\uV^{\rm sBDFM}_{h0}$.

Define a nodal interpolation operator $\Pi_h:\undertilde{H}^1(\Omega)\rightarrow \undertilde{V}^{\rm sBDFM}_h$ such that for any $e \subset \mathcal{E}_h$,
\begin{equation}\label{Pi_h}\nonumber
	\fint_e (\Pi_h \undertilde{v} \cdot \textbf{n}_e) p = \fint_e (\undertilde{v} \cdot \textbf{n}_e) p, \ \forall\,p\in P_1(e)\ \ \ \mbox{and}\ \  \fint_e \Pi_h \undertilde{v} \cdot \textbf{t}_e  = \fint_e \undertilde{v} \cdot \textbf{t}_e .
\end{equation}
The operator $\Pi_h$ is locally defined on each triangle, and it preserves linear functions locally. Furthermore, the local space $\undertilde{V}_h(T)$ restricted on $T$ is invariant under the Piola's transformation, i.e., it maps $\undertilde{V}_h(T)$ onto $\undertilde{V}_h(\hat{T})$. Therefore, approximation estimates of $\Pi_h$ can be derived from standard scaling arguments and the Bramble-Hilbert lemma.

\begin{proposition}
	It holds for $0 \leq k \leq 1 \leq s \leq 3$ that
	\begin{equation}
		|\undertilde{v} - \Pi_h \undertilde{v}|_{k,h}\leqslant C h^{s-k} |\undertilde{v}|_{s,\Omega}, \quad \forall\, \undertilde{v}\in \undertilde{H}^s(\Omega).
	\end{equation}
\end{proposition}

\subsection{Structure of the kernel of  $\dv$ on a closed patch}

For an $m-$cell interior vertex patch $P_O$, we label cells of it sequentially as $T_i,i=1:m$, and label $e_i = \overline{T_i}\cap \overline{T_{i+1}},i=1:m-1, e_m = \overline{T_m}\cap \overline{T_{1}}$. Also, we label $e_{m+i},i=1:m$, the edge opposite $O$ in $T_i$; see Figure \ref{fig:patchillu} (left) for an illustration.

Viewing $P_O$ as a special grid, we construct $\uV{}^{\rm sBDFM}_{h0}(P_O)$ thereon, and denote 
$$
\uZ{}_{O}:=\{\uv\in \uV{}^{\rm sBDFM}_{h0}(P_O):\dv\,\uv=0\}.
$$

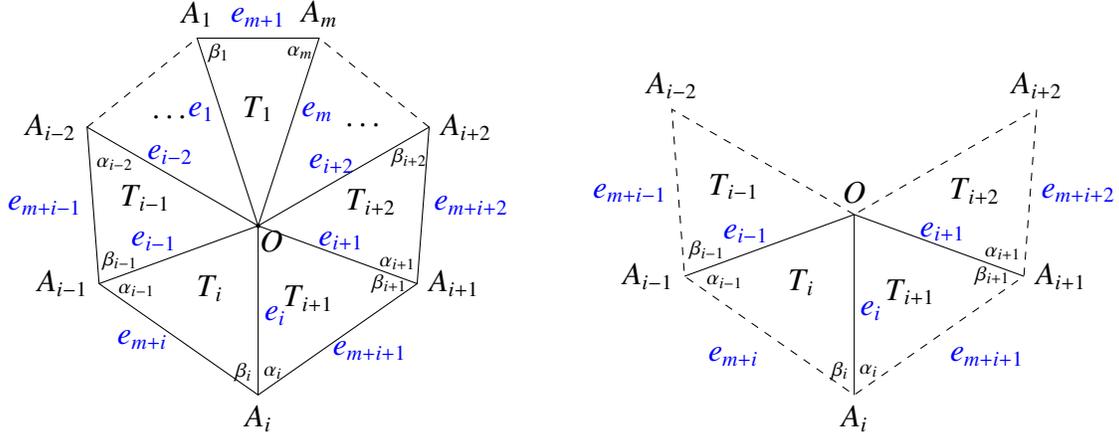
\begin{figure}[htbp]
	\centering
	\begin{tikzpicture}[scale=0.75]
		\path coordinate (O) at (0,0)
		coordinate (A) at (-90:3)
		coordinate (B) at (-160:3)
		coordinate (C) at (-20:3)
		coordinate (D) at (150:3.5)
		coordinate (E) at (30:3.5)
		coordinate (F) at (108:3.5)
		coordinate (G) at (72:3.5);
		
		\foreach \i/\angle in {O/-80} {
			\node[dot, label={[dot label]\angle:$\i$}] at (\i) {};
		}

		\node [below] at (A) {$A_i$};
		\node [left] at (B) {$A_{i-1}$};
		\node [right] at (C) {$A_{i+1}$};
		\node [left] at (D) {$A_{i-2}$};
		\node [right] at (E) {$A_{i+2}$};
		\node [above] at (F) {$A_1$};
		\node [above] at (G) {$A_m$};
		
		\draw[line width=.4pt]  (O) -- (A) ;
		\draw[line width=.4pt]  (O) -- (B) ;
		\draw[line width=.4pt]  (O) -- (C) ;
		\draw[line width=.4pt]  (O) -- (D) -- (B) -- (A) -- (C) -- (E) -- cycle ;
		\draw[line width=.4pt]  (O) -- (F) -- (G) -- cycle ;
		\draw[dashed] (D) -- (F) ;
		\draw[dashed] (E) -- (G) ;
		
		\path coordinate (B1) at ($(A)!.5!(B)$)
		coordinate (B2) at ($(A)!.4!(C)$)
		coordinate (B3) at ($(B)!.5!(D)$)
		coordinate (B4) at ($(C)!.5!(E)$)
		coordinate (B5) at ($(G)!.5!(F)$);
		
		\node [left,blue] at (B1) {$ e_{m+i} $};
		\node [right,blue] at (B2) {$ e_{m+i+1} $};
		\node [left,blue] at (B3) {$ e_{m+i-1} $};
		\node [right,blue] at (B4) {$ e_{m+i+2} $};
		\node [above,blue] at (B5) {$ e_{m+1} $};
		
		\path coordinate (C1) at (-93:1.6)
		coordinate (C2) at (-153:1.4)
		coordinate (C3) at (-23:1.6)
		coordinate (C4) at (150:1.8)
		coordinate (C5) at (30:2.2)
		coordinate (C6) at (106:2.1)
		coordinate (C7) at (74:2.1);
		
		\node[right,blue] at (C1) {$ e_i $};
		\node[above left,blue] at (C2) {$ e_{i-1} $};
		\node[above,blue] at (C3) {$ e_{i+1} $};
		\node[above,blue] at (C4) {$ e_{i-2} $};
		\node[left,blue] at (C5) {$ e_{i+2} $};
		\node[left,blue] at (C6) {$ e_1 $};
		\node[right,blue] at (C7) {$ e_{m} $};

		\path coordinate (A1) at (-130:2)
		coordinate (A2) at (-40:2)
		coordinate (A3) at (-182:2)
		coordinate (A4) at (0:2)
		coordinate (A5) at (90:2.5)
		coordinate (A6) at (134:2.2)
		coordinate (A7) at (52:2.2);
		
		\node[above right] at (A1) {$ T_i $};
		\node[left] at (A2) {$ T_{i+1} $};
		\node[above] at (A3) {$ T_{i-1} $};
		\node[above] at (A4) {$ T_{i+2} $};
		\node[below] at (A5) {$ T_1 $};
		\node[above] at (A6) {$ \cdots $};
		\node[right] at (A7) {$ \cdots $};
		
		\path coordinate (D1) at (-92:2.6)
		coordinate (D2) at (-88:2.6)
		coordinate (D3) at (-159:2.3)
		coordinate (D4) at (-174:2.45)
		coordinate (D5) at (-8:2.5)
		coordinate (D6) at (-29:2.1)
		coordinate (D7) at (151:2.9)
		coordinate (D8) at (29:2.5)
		coordinate (D9) at (84:3.1)
		coordinate (D10) at (109:3.25);
		
		\node [right,font=\tiny] at (D1) {$\alpha_i$};
		\node [left,font=\tiny] at (D2) {$\beta_i$};
		\node [below,font=\tiny] at (D3) {$\alpha_{i-1}$};
		\node [below,font=\tiny] at (D4) {$\beta_{i-1}$};
		\node [below,font=\tiny] at (D5) {$\alpha_{i+1}$};
		\node [right,font=\tiny] at (D6) {$\beta_{i+1}$};
		\node [below,font=\tiny] at (D7) {$\alpha_{i-2}$};
		\node [right,font=\tiny] at (D8) {$\beta_{i+2}$};
		\node [right,font=\tiny] at (D9) {$\alpha_{m}$};
		\node [right,font=\tiny] at (D10) {$\beta_1$};
	\end{tikzpicture}
	\hspace{0.25in}
	\begin{tikzpicture}[scale=0.8]
		\path coordinate (O) at (0,0)
		coordinate (A) at (-90:3)
		coordinate (B) at (-160:3)
		coordinate (C) at (-20:3)
		coordinate (D) at (150:3.5)
		coordinate (E) at (30:3.5);
		\node[above] at (O) {$O$};
		\node [below] at (A) {$A_i$};
		\node [left] at (B) {$A_{i-1}$};
		\node [right] at (C) {$A_{i+1}$};
		\node [above] at (D) {$A_{i-2}$};
		\node [above] at (E) {$A_{i+2}$};
		
		\draw[line width=.4pt]  (O) -- (A) ;
		\draw[line width=.4pt]  (O) -- (B) ;
		\draw[line width=.4pt]  (O) -- (C) ;
		\draw[dashed]  (O) -- (D) -- (B) -- (A) -- (C) -- (E) -- cycle ;
		
		\path coordinate (B1) at ($(A)!.5!(B)$)
		coordinate (B2) at ($(A)!.5!(C)$)
		coordinate (B3) at ($(B)!.5!(D)$)
		coordinate (B4) at ($(C)!.5!(E)$);
		
		\node [below left,blue] at (B1) {$ e_{m+i} $};
		\node [below right,blue] at (B2) {$ e_{m+i+1} $};
		\node [left,blue] at (B3) {$ e_{m+i-1} $};
		\node [right,blue] at (B4) {$ e_{m+i+2} $};

		\path coordinate (C1) at (-93:1.6)
		coordinate (C2) at (-152:1.4)
		coordinate (C3) at (-23:1.6);
		
		\node[right,blue] at (C1) {$ e_i $};
		\node[above left,blue] at (C2) {$ e_{i-1} $};
		\node[above,blue] at (C3) {$ e_{i+1} $};

		\path coordinate (A1) at (-130:2)
		coordinate (A2) at (-40:2)
		coordinate (A3) at (-182:2)
		coordinate (A4) at (0:2);
		
		\node[above right] at (A1) {$ T_i $};
		\node[left] at (A2) {$ T_{i+1} $};
		\node[above] at (A3) {$ T_{i-1} $};
		\node[above] at (A4) {$ T_{i+2} $};
		
		\path coordinate (D1) at (-92:2.6)
		coordinate (D2) at (-88:2.6)
		coordinate (D3) at (-159:2.3)
		coordinate (D4) at (-174:2.45)
		coordinate (D5) at (-8:2.5)
		coordinate (D6) at (-29:2.1);
		
		\node [right,font=\tiny] at (D1) {$\alpha_i$};
		\node [left,font=\tiny] at (D2) {$\beta_i$};
		\node [below,font=\tiny] at (D3) {$\alpha_{i-1}$};
		\node [below,font=\tiny] at (D4) {$\beta_{i-1}$};
		\node [below,font=\tiny] at (D5) {$\alpha_{i+1}$};
		\node [right,font=\tiny] at (D6) {$\beta_{i+1}$};	
	\end{tikzpicture}
	\caption{Illustration of a patch around $O$(left) and its part amplification(right).}\label{fig:patchillu}
\end{figure}
\begin{lemma}\label{lem:dimker=1}
	\label{dim_P_O=1}
	$\dim(\uZ{}_{O})=1$.
\end{lemma}
\begin{proof}
	Assume $\upsi{}_h\in \uZ{}_{O}$, then $\upsi{}_h|_{T_i}\subset\uZ{}_{T_i}$, $i=1:m$. By the boundary conditions, it follows that
	\begin{equation}\label{eq:restri}
		\upsi{}_h|_{T_i}=\gamma_{T_i}^{i-1}\uw{}_{T_i,e_{i-1}}+\gamma_{T_i}^{i}\uw{}_{T_i,e_{i}}+\gamma_{T_i}^{i-1,i}\uw{}_{T_i,e_{i-1}e_i},
	\end{equation}
	with $\gamma_{T_i}^{i-1}$, $\gamma_{T_i}^{i}$ and $\gamma_{T_i}^{i-1,i}$ determined such that $\upsi_h$ satisfies the continuity restriction of $\uV{}^{\rm sBDFM}_{h}$.

	For an arbitrary edge $e_i$, $1\leqslant i\leqslant m$, across it the normal component of $\upsi_h$ and integration of the tangential component of $\upsi_h$ are continuous; see Figure \ref{fig:patchillu}(right) for an illustration. Based on the continuity conditions, a direct calculation shows that
	\begin{equation}
		\left\{
		\begin{aligned}
			&\gamma^{i-1,i}_{T_i}=\gamma_{T_{i+1}}^{i,i+1},\\
			& \gamma_{T_i}^{i} = \frac{d_{m+i} d_{m+i+1} \sin{(\alpha_i + \beta_i)}}{2(S_i+S_{i+1})}   \gamma^{i-1,i}_{T_i},\\
			&\gamma_{T_{i+1}}^{i} = \frac{d_{m+i} d_{m+i+1} \sin{(\alpha_i + \beta_i)}}{2(S_i+S_{i+1})}  \gamma_{T_{i+1}}^{i,i+1}.
		\end{aligned}
		\right.
	\end{equation}
	By checking all edges $e_i$, $i=1:m$, we have
	\begin{equation}\label{eq2 patch dominant}
		\gamma^{m,1}_{T_1}=\gamma^{1,2}_{T_2}=...=\gamma^{m-1,m}_{T_m},
	\end{equation}
	and
	\begin{equation}\label{dominanted coes}
		\left\{
		\begin{aligned}
			&  \gamma_{T_i}^{i-1} = \frac{d_{m+i-1} d_{m+i} \sin{(\alpha_{i-1} + \beta_{i-1})}}{2(S_{i-1}+S_i)} \gamma^{i-1,i}_{T_i},\\
			&  \gamma_{T_i}^{i}= \frac{d_{m+i} d_{m+i+1} \sin{(\alpha_i + \beta_i)}}{2(S_i+S_{i+1})} \gamma^{i-1,i}_{T_i}.
		\end{aligned}
		\right.
	\end{equation}
	In other words,
	\begin{equation}\label{dominanted coes1}
		\left\{
		\begin{aligned}
			& \gamma_{T_1}^{m} = \frac{d_{2m} d_{m+1} \sin{(\alpha_{m} + \beta_{m})}}{2(S_{m}+S_1)} \gamma^{m,1}_{T_1} , \\
			& \gamma_{T_i}^{i-1} = \frac{d_{m+i-1} d_{m+i} \sin{(\alpha_{i-1} + \beta_{i-1})}}{2(S_{i-1}+S_i)} \gamma^{i-1,i}_{T_i}  \ (i=2:m),
		\end{aligned}
		\right.
	\end{equation}
	and
	\begin{equation}\label{dominanted coes2}
		\left\{
		\begin{aligned}
			& \gamma_{T_i}^{i}= \frac{d_{m+i} d_{m+i+1} \sin{(\alpha_i + \beta_i)}}{2(S_i+S_{i+1})} \gamma^{i-1,i}_{T_i}  \ (i=1:m-1) ,\\
			& \gamma_{T_m}^{m}= \frac{d_{2m} d_{m+1} \sin{(\alpha_m + \beta_m)}}{2(S_m+S_1)} \gamma^{m-1,m}_{T_m}. \\
		\end{aligned}
		\right.
	\end{equation}
	
	By the relations \eqref{eq2 patch dominant}, \eqref{dominanted coes1} and \eqref{dominanted coes2}, we can choose $\gamma^{m,1}_{T_1}=1$, and other coefficients accordingly, then $\upsi{}_h\in \uZ{}_O$, and moreover $\uZ{}_O={\rm span}\{\upsi{}_h\}$. The proof is completed.
\end{proof}

\subsection{A stable conservative pair for the Stokes problem}

Denote
$$
\mathbb{P}{}^1_h(\mathcal{T}_h):=\{q_h\in L^2(\Omega):q_h|_T\in P_1(T), \forall\, T \in \mathcal{T}_h\} \ \mbox{and}\ \ \mathbb{P}^1_{h0}(\mathcal{T}_h):=\mathbb{P}{}^1_h(\mathcal{T}_h)\cap L^2_0(\Omega).
$$
Then $\uV^{\rm sBDFM}_{h0}\times \mathbb{P}^1_{h0}$ forms a stable pair for the Stokes problem.

\begin{theorem}[Inf-sup conditions]\label{thm:infsupsBDFM}
	Let $\{\mathcal{T}_h\}$ be a family of triangulations of $\Omega$ satisfying \emph{\textbf{Assumption\,\ref{assumption grid}}}. Then
	\begin{equation}\label{inf-sup homogeneous}
		\sup_{\undertilde{v}_h \in \undertilde{V}^{\rm sBDFM}_{h0}} \frac{( div \ \undertilde{v}_h , q_h )}{||\undertilde{v}_h||_{1,h}} \geqslant C ||q_h||_{0,\Omega},
		\forall\, q_h\in \mathbb{P}^1_h(\mathcal{T}_h).
	\end{equation}
\end{theorem}
\begin{proof}
	Firstly, for any interior vertex $O$ and its patch $P_O$, we can construct $\uV^{\rm sBDFM}_{h0}(P_O)$ and $\mathbb{P}^1_{h0}(P_O)$. Obviously $\dv\, \uV^{\rm sBDFM}_{h0}(P_O)\subset\mathbb{P}^1_{h0}(P_O)$. Thus by counting the dimension, we obtain $\dv\, \uV^{\rm sBDFM}_{h0}(P_O)=\mathbb{P}^1_{h0}(P_O)$ by Lemma \ref{lem:dimker=1}. This verifies the condition (C${}_1$) of Proposition \ref{macroelement technique}. 
	
	The other conditions of Proposition \ref{macroelement technique} are direct, and the inf-sup condition holds by Proposition \ref{macroelement technique}. The proof is completed. 
\end{proof}

Now we consider the finite element discretization: Find $(\uphi{}_h,p_h)\in \uV^{\rm sBDFM}_{h0}\times \mathbb{P}^1_{h0}$, such that
\begin{equation}\label{eq:dissbdfm}
	\left\{
	\begin{array}{lcll}
		\varepsilon^2(\nabla_h \, \uphi{}_h,\nabla_h\, \upsi{}_h)+(\dv\,\upsi{}_h,p_h)&=&(\uf,\upsi{}_h),&\forall\,\upsi{}_h\in \uV^{\rm sBDFM}_{h0}
		\\
		(\dv\,\uphi{}_h,q_h)&=&0,&\forall\,q_h\in \mathbb{P}^1_{h0}.
	\end{array}
	\right.
\end{equation}
The well-posedness of \eqref{eq:dissbdfm} is immediate. 
\begin{lemma}\label{lem:kernelappr}
	Given $\uphi\in \uH^1_0(\Omega)\cap \uH^2(\Omega)$ such that $\dv\,\uphi=0$, it holds that
	\begin{equation}
		\inf_{\upsi_h\in \uV^{\rm sBDFM}_{h0},\,\dv\,\upsi_h=0}\|\uphi-\upsi{}_h\|_{1,h} \leqslant C h \|\uphi\|_{2,\Omega}.
	\end{equation}
\end{lemma}
\begin{proof}
	Let $(\uphi^*,p^*)\in \uH^1_0(\Omega)\times L^2_0(\Omega)$ be such that
	\begin{equation}
		\left\{
		\begin{array}{lcll}
			(\nabla\, \uphi^*,\nabla\,\upsi)+(p^*,\dv\,\upsi)&=&(\curl\,\rot\,\uphi,\upsi),&\forall\,\upsi\in \uH{}^1_0(\Omega),
			\\
			(\dv\,\uphi^*,q)&=&0,&\forall\,q\in L^2_0(\Omega).
		\end{array}
		\right.
	\end{equation}
	Then $\uphi^*=\uphi$ and $p=0$.
	Now let $(\uphi^*_h,p^*_h)\in \uV{}^{\rm sBDFM}_{h0}\times \mathbb{P}^1_{h0}$ be such that
	\begin{equation}
		\left\{
		\begin{array}{lcll}
			(\nabla_h \,  \uphi^*_h,\nabla_h\,\upsi_h)+(\dv\,\upsi_h,p^*_h)&=&(\curl\,\rot\,\uphi,\upsi_h),&\forall\,\upsi_h\in \uV{}^{\rm sBDFM}_{h0},
			\\
			(\dv\,\uphi^*_h,q_h)&=&0,&\forall\,q_h\in \mathbb{P}^1_{h0}.
		\end{array}
		\right.
	\end{equation}
	Then $\dv\,\uphi^*_h=0$ and $\|\uphi^*-\uphi^*_h\| \leqslant C h\|\uphi\|_{2,\Omega}$. The proof is completed.
\end{proof}

The convergence estimate robust in $\varepsilon$ can be obtained in a standard way. 
\begin{theorem}\label{thm:convsBDFM}
	Let $(\uphi,p)$ and $(\uphi{}_h,p_h)$ be the solutions of \eqref{eq:Stokes vf} and \eqref{eq:dissbdfm}, respectively. If $(\uphi,p)\in \uH^2(\Omega)\times H^1(\Omega)$, then
	\begin{equation}
		|\uu - \uu{}_{h}|_{1,h}  \leqslant C h|\uu|_{2,\Omega},\ \ \ \mbox{and}\ \ \   \|p-p_{h}\|_{0,\Omega} \leqslant C (h|p|_{1,\Omega} + \varepsilon^{2} h|\uu|_{2,\Omega}). 
	\end{equation}
\end{theorem}

\section{A continuous nonconforming finite element scheme for the biharmonic equation}
\label{sec:bh}

\subsection{A finite element Stokes complex}

Define
\begin{equation}
	V_h^{2+}:=\{v_h\in H^1(\Omega):v_h|_T\in P^{2+}(T),\ \forall\,T\in\mathcal{T}_h;\ \int_e\frac{\partial v_h}{\partial\mathbf{n}}\ \mbox{is\ continuous\ across\ interior\ edge}\ e\},
\end{equation}
and
\begin{equation}\label{eq:vh02+}
	V^{2+}_{h0}:=\{v_h\in V_h^{2+}\cap H^1_0(\Omega):\int_e\frac{\partial v_h}{\partial\mathbf{n}}=0 \ \mbox{on\ boundary\ edge}\ e\}.
\end{equation}

\begin{lemma}\label{lem:essbdfm}
	The exact sequence holds
	\begin{equation}
		\{0\} \xrightarrow{\rm inc} V^{2+}_{h0} \xrightarrow{{\curl}} \uV{}^{\rm sBDFM}_{h0} \xrightarrow{{\dv}} \mathbb{P}^1_{h0}  \xrightarrow{\int_\Omega\cdot} \{0\}.
	\end{equation}
\end{lemma}
\begin{proof}
	Regarding Theorem \ref{thm:infsupsBDFM}, we only have to show
	\begin{equation}
		\{\uv_h\in \uV{}^{\rm sBDFM}_{h0}:\dv\,\uv_h=0\}=\curl\, V^{2+}_{h0}.
	\end{equation}
	Denote $V^{2+,C}_{h0}:=\{v_h\in H^1_0(\Omega):v_h|_T\in P^{2+}(T),\ \forall\,T\in\mathcal{T}_h\}$. Given $\uv{}_h\in \uV{}^{\rm sBDFM}_{h0}\subset H_0(\dv,\Omega)$ such that $\dv\,\uv{}_h=0$, by the local exact sequence Lemma \ref{lem:locales} and the de Rham complex \ref{eq:derhamc}, there exists a $w_h\in V^{2+,C}_{h0}$, such that $\curl\, w_h=\uv{}_h$. Further, by the tangential continuity restriction on $\uv{}_h$, it follows that $w_h\in V^{2+}_{h0}$. The proof is completed.
\end{proof}

\subsection{A low-degree scheme for biharmonic equation}

We consider the biharmonic equation: given $g\in H^{-1}(\Omega)$, find $u\in H^2_0(\Omega)$, such that
\begin{equation}\label{eq:bih}
	(\nabla^2\,u,\nabla^2\,v)=\langle g,v\rangle ,\quad\forall\,v\in H^2_0(\Omega).
\end{equation}
A finite element discretization is to find $u_h\in V^{2+}_{h0}$, such that
\begin{equation}\label{eq:bihdis}
	(\nabla_h^2\,u_h,\nabla_h^2\,v_h)=\langle g,v_h\rangle,\quad\forall\,v_h\in V^{2+}_{h0}.
\end{equation}

The lemma below is an immediate consequence of Lemmas \ref{lem:kernelappr} and \ref{lem:essbdfm}.
\begin{lemma}
	It holds for $w\in H^3(\Omega)\cap H^2_0(\Omega)$ that
	\begin{equation}
		\inf_{v_h\in V^{2+}_{h0}}\|w-v_h\|_{2,h}\leqslant C h\|w\|_{3,\Omega}.
	\end{equation}
\end{lemma}
\begin{proof}
	By Lemmas \ref{lem:essbdfm} and \ref{lem:kernelappr},
	\begin{multline}
		\qquad\inf_{v_h\in V^{2+}_{h0}}|w-v_h|_{2,h}=\inf_{v_h\in V^{2+}_{h0}}|\curl\, w-\curl\, v_h|_{1,h}
		\\
		=\inf_{\upsi_h\in \uV^{\rm sBDFM}_{h0},\,\dv\,\upsi_h=0}|\curl\, w-\upsi_h|_{1,h}\leqslant C h|\curl\, w|_{2,\Omega}\leqslant C h\|w\|_{3,\Omega}.\qquad
	\end{multline}
	This completes the proof. 
\end{proof}
\begin{theorem}
	Let $u$ and $u_h$ be the solutions of \eqref{eq:bih} and \eqref{eq:bihdis} respectively, and assume $u\in H^3(\Omega)\cap H^2_0(\Omega)$. Then
	\begin{equation}
		\|u-u_h\|_{2,h}\leqslant C h\|u\|_{3,\Omega}.
	\end{equation}
\end{theorem}

The proof of the theorem follows from standard arguments, and we omit it here. 

\subsection{Basis functions of $V^{2+}_{h0}$}

For the implementation of the finite element schemes, in this section, we present the explicit formulation of basis functions of certain finite element spaces.

\subsubsection{Basis function of the kernel subspace of {\rm sBDFM} element}
Denote the kernel subspace
\begin{equation}
	\uZ_{h0}:=\{\uv_h\in \uV^{\rm sBDFM}_{h0}:\dv\,\uv_h=0\}.
\end{equation}


\begin{figure}[htbp]
	\centering
	\begin{tikzpicture}[scale=0.8]
		\path coordinate (O) at (0,0)
		coordinate (A) at (-90:3)
		coordinate (B) at (-160:3)
		coordinate (C) at (-20:3)
		coordinate (D) at (150:3.5)
		coordinate (E) at (30:3.5)
		coordinate (F) at (108:3.5)
		coordinate (G) at (72:3.5);
		
		\foreach \i/\angle in {O/-80} {
			\node[dot, label={[dot label]\angle:$\i$}] at (\i) {};
		}

		\node [below] at (A) {$A_1$};
		\node [left] at (B) {$A_m$};
		\node [right] at (C) {$A_2$};
		\node [left] at (D) {$A_{m-1}$};
		\node [right] at (E) {$A_3$};
		\node [above] at (F) {$A_i$};
		\node [above] at (G) {$A_{i-1}$};
		
		\draw[line width=.4pt]  (O) -- (A) ;
		\draw[line width=.4pt]  (O) -- (B) ;
		\draw[line width=.4pt]  (O) -- (C) ;
		\draw[line width=.4pt]  (O) -- (D) -- (B) -- (A) -- (C) -- (E) -- cycle ;
		\draw[line width=.4pt]  (O) -- (F) -- (G) -- cycle ;
		\draw[dashed] (D) -- (F) ;
		\draw[dashed] (E) -- (G) ;
		
		\path coordinate (B1) at ($(A)!.5!(B)$)
		coordinate (B2) at ($(A)!.4!(C)$)
		coordinate (B3) at ($(B)!.5!(D)$)
		coordinate (B4) at ($(C)!.5!(E)$)
		coordinate (B5) at ($(F)!.5!(G)$);
		
		\node [left,blue] at (B1) {$ e_{m+1} $};
		\node [right,blue] at (B2) {$ e_{m+2} $};
		\node [left,blue] at (B3) {$ e_{2m} $};
		\node [right,blue] at (B4) {$ e_{m+3} $};
		\node [above,blue] at (B5) {$ e_{m+i} $};
		
		\path coordinate (C1) at (-93:1.6)
		coordinate (C2) at (-153:1.4)
		coordinate (C3) at (-23:1.6)
		coordinate (C4) at (150:1.8)
		coordinate (C5) at (30:2.2)
		coordinate (C6) at (106:2.1)
		coordinate (C7) at (74:2.1);
		
		\node[right,blue] at (C1) {$ e_1 $};
		\node[above left,blue] at (C2) {$ e_m $};
		\node[above,blue] at (C3) {$ e_2 $};
		\node[above,blue] at (C4) {$ e_{m-1} $};
		\node[left,blue] at (C5) {$ e_3 $};
		\node[left,blue] at (C6) {$ e_i $};
		\node[right,blue] at (C7) {$ e_{i-1} $};

		\path coordinate (A1) at (-130:2)
		coordinate (A2) at (-40:2)
		coordinate (A3) at (-182:2)
		coordinate (A4) at (0:2)
		coordinate (A5) at (90:2.5)
		coordinate (A6) at (134:2.2)
		coordinate (A7) at (52:2.2);
		
		\node[above right] at (A1) {$ T_1 $};
		\node[left] at (A2) {$ T_2 $};
		\node[above] at (A3) {$ T_m $};
		\node[above] at (A4) {$ T_3 $};
		\node[below] at (A5) {$ T_i $};
		\node[above] at (A6) {$ \cdots $};
		\node[right] at (A7) {$ \cdots $};
		
		\path coordinate (D1) at (-90:2.5)
		coordinate (D2) at (-89:2.5)
		coordinate (D3) at (-159:2.45)
		coordinate (D4) at (-173:2.5)
		coordinate (D5) at (-10:2.6)
		coordinate (D6) at (-29:2.3)
		coordinate (D7) at (151:2.95)
		coordinate (D8) at (28:2.7)
		coordinate (D9) at (88:3.1)
		coordinate (D10) at (108:3.25);
		
		\node [right,font=\tiny] at (D1) {$\alpha_1$};
		\node [left,font=\tiny] at (D2) {$\beta_1$};
		\node [below,font=\tiny] at (D3) {$\alpha_m$};
		\node [below,font=\tiny] at (D4) {$\beta_m$};
		\node [below,font=\tiny] at (D5) {$\alpha_2$};
		\node [right,font=\tiny] at (D6) {$\beta_2$};
		\node [below,font=\tiny] at (D7) {$\alpha_{m-1}$};
		\node [right,font=\tiny] at (D8) {$\beta_3$};
		\node [right,font=\tiny] at (D9) {$\alpha_{i-1}$};
		\node [right,font=\tiny] at (D10) {$\beta_i$};
	\end{tikzpicture}
	\hspace{0.3in}
	\begin{tikzpicture}[scale=0.8]
		\path coordinate (A) at (0,0)
		coordinate (B) at (60:3)
		coordinate (C) at (120:3)
		coordinate (E) at (180:3)
		coordinate (F) at (0:3)
		coordinate (D) at ($ (C)!-1!(E) $);
		
		\draw[line width=.4pt]  (A) -- (B) -- (C) -- cycle;
		\draw[line width=.4pt]  (A) -- (F) -- (B) -- (D) -- (C) -- (E)-- cycle;
		
		\path coordinate (A1) at ($ (B)!.5!(C) $)
		coordinate (B1) at ($ (A)!.5!(C) $)
		coordinate (C1) at ($ (A)!.5!(B) $)
		coordinate (A2) at ($ (A1)!.4!(C) $)
		coordinate (B2) at ($ (B1)!.4!(A) $)
		coordinate (C2) at ($ (C1)!.4!(B) $)
		coordinate (A3) at ($ (A1)!0.6cm!90:(B) $)
		coordinate (B3) at ($ (B1)!0.6cm!90:(C) $)
		coordinate (C3) at ($ (C1)!0.6cm!90:(A) $)
		coordinate (O) at ($ (A1)!0.66cm!-90:(B) $);
		
		\path   coordinate (C4) at ($ (C)!0.25cm!-10:(A) $)
		coordinate (B4) at ($ (B)!0.5cm!-2:(C) $)
		coordinate (A4) at ($ (A)!0.67cm!1:(B) $)
		coordinate (D1) at ($ (D)!.35!(A1) $)
		coordinate (E1) at ($ (E)!.35!(B1) $)
		coordinate (F1) at ($ (F)!.35!(C1) $);

		\node [below] at (O) {$T_0$};
		\node [below] at (D1) {$T_1$};
		\node [right] at (E1) {$T_2$};
		\node [left] at (F1) {$T_3$};

		\node[below] at (A) {$A_1$};
		\node[right] at (B) {$A_2$};
		\node[left] at (C) {$A_3$};
		\node[above] at (D) {$A_4$};
		\node[left] at (E) {$A_5$};
		\node[right] at (F) {$A_6$};
		
		\node[above,blue] at (A1) {$e_1$};
		\node[left,blue] at (B1) {$e_2$};
		\node[right,blue] at (C1) {$e_3$};
		
		\node[right,font=\tiny] at (C4) {$\alpha_3$};
		\node[below,font=\tiny] at (B4) {$\alpha_2$};
		\node[left,font=\tiny] at (A4) {$\alpha_1$};
	\end{tikzpicture}
	\caption{Illustration of a patch around $O$(left) and a patch around $T_0$(right).}
	\label{fig:basissupp}
\end{figure}
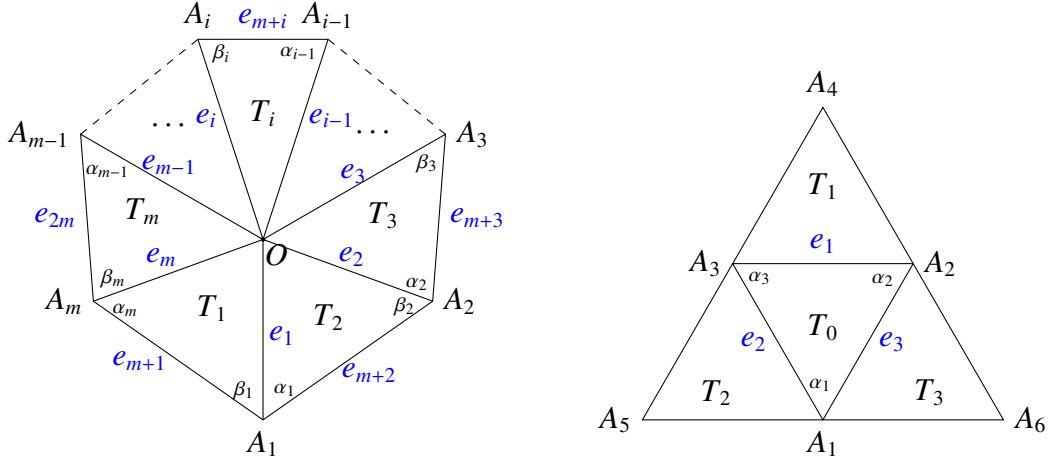

Firstly, associated with the interior vertex patch around an interior vertex $O$(cf. Figure \ref{fig:basissupp}, left), denote $\upsi^O$ as
\begin{equation}\label{math expr:vertex-centered basis}
	\upsi^O=
	\left\{
	\begin{aligned}
		& \frac{d_{m+1}d_{m+2} \sin{(\alpha_1 + \beta_1)}}{2(S_1 + S_2)} \uw_{T_1,e_1}
		+ \frac{d_{m+1}d_{2m} \sin{(\alpha_m + \beta_m)}}{2(S_1 + S_m)} \uw_{T_1,e_m}
		+  \uw_{T_1,e_1,e_m}, \ in \ T_1,\\
		& \frac{d_{m+i}d_{m+i+1} \sin{(\alpha_i + \beta_i)}}{2(S_i + S_{i+1})} \uw_{T_i,e_i}
		+ \frac{d_{m+i}d_{m+i-1} \sin{(\alpha_{i-1} + \beta_{i-1})}}{2(S_i + S_{i-1})} \uw_{T_i,e_{i-1}}
		+ \uw_{T_i,e_i,e_{i-1}}, \ in \ T_i(i=2:m-1),\\
		& \frac{d_{2m}d_{m+1} \sin{(\alpha_m + \beta_m)}}{2(S_m + S_1)} \uw_{T_m,e_m}
		+ \frac{d_{2m} d_{2m-1} \sin{(\alpha_{m-1} + \beta_{m-1})}}{2 (S_m+S_{m-1})} \uw_{T_m,e_{m-1}}
		+ \uw_{T_m,e_m,e_{m-1}}, \ in \ T_m.
	\end{aligned}
	\right.
\end{equation}
Secondly, associated with the interior cell patch around an interior cell $T_0$ (cf. Figure \ref{fig:basissupp}, right), denote $\upsi_{T_0}$ as
\begin{equation}\label{math expr:Tpatch}
	\undertilde{\psi}_{T_0} =
	\left\{
	\begin{aligned}
		& \frac{S_1}{S_1+S_0} \undertilde{w}_{T_1, e_1}, \ in \ T_1,\\
		& \frac{S_2}{S_2+S_0} \undertilde{w}_{T_2, e_2}, \ in \ T_2,\\
		& \frac{S_3}{S_3+S_0} \undertilde{w}_{T_3, e_3}, \ in \ T_3,\\
		& \frac{1}{3}(
		\frac{S_1-2S_0}{S_1+S_0} \undertilde{w}_{T_0,e_1}
		+ \frac{S_2-2S_0}{S_2+S_0} \undertilde{w}_{T_0,e_2}
		+ \frac{S_3-2S_0}{S_3+S_0} \undertilde{w}_{T_0,e_3}
		+  \undertilde{w}_{T_0,e_2,e_3}
		+  \undertilde{w}_{T_0,e_3,e_1}
		+  \undertilde{w}_{T_0,e_1,e_2}
		), \ in \ T_0.
	\end{aligned}
	\right.
\end{equation}

Given an interior cell $T_0$ with vertices $A_i$, $i=1:3$, and neighbored cells $T_j$, $j=1:3$, the cell $T_0$ is covered by $\upsi^{A_i}|_{T_0}$ for $i=1:3$ and $\upsi_{T_j}|_{T_0}$ for $j=0:3$; see Figure \ref{fig:Tpatch_grid} for an illustration. It is easy to know $\{\upsi^{A_i}|_{T_0},\ i=1:3,\ \upsi_{T_j}|_{T_0},\ j=0:3\}$ are linearly dependent. However, any six of them are linearly independent. For conciseness, we show the following lemma.

\begin{lemma}\label{independence on boundary-cell}
	For an interior cell $T_0$ with vertices $A_i,i=1:3$, and neighbored cells $T_j,j=1:3$,(cf. Figure \ref{fig:Tpatch_grid}) the functions $\{\upsi^{A_i}|_{T_0},\ i=2:3,\ \upsi_{T_j}|_{T_0},\ j=0:3\}$ are linearly independent.
\end{lemma}
\begin{proof}
	A direct calculation leads to
	$$
	\left(
	\upsi^{A_2}|_{T_0},
	\upsi^{A_3}|_{T_0},
	\upsi_{T_0}|_{T_0},
	\upsi_{T_1}|_{T_0},
	\upsi_{T_2}|_{T_0},
	\upsi_{T_3}|_{T_0}
	\right)^\top
	=
	\mathbf{A}
	\left(
	\uw{}_{T_0,e_2,e_3},
	\uw{}_{T_0,e_3,e_1},
	\uw{}_{T_0,e_1,e_2},
	\uw{}_{T_0,e_1},
	\uw{}_{T_0,e_2},
	\uw{}_{T_0,e_3}\right)^\top
	$$
	 
	$$
	\mbox{with}\ \ \mathbf{A}=
	\left[\begin{array}{cccccc}
		0 &  1 &                0 &  \frac{d_2 d_5 \sin{(\alpha_3+\gamma_3)}}{2(S_1+S_0)}   & 0 &  \frac{d_2 d_8 \sin{(\alpha_1+\beta_1)}}{2(S_3+S_0)}   \\                
		0 &                0 &   1  &  \frac{d_3 d_4 \sin{(\alpha_2+\beta_2)}}{2(S_1+S_0)}   &  \frac{d_3 d_7 \sin{(\alpha_1+\gamma_1)}}{2(S_2+S_0)}   & 0 \\
		\frac{1}{3}  &   \frac{1}{3}  &   \frac{1}{3}  &  \frac{S_1-2S_0}{3(S_1+S_0)}   &  \frac{S_2-2S_0}{3(S_2+S_0)}   &  \frac{S_3-2S_0}{3(S_3+S_0)}   \\
		0 &                0 &                0 &  \frac{S_0}{S_1+S_0}   & 0 & 0 \\
		0 &                0 &                0 & 0 &  \frac{S_0}{S_2+S_0}   & 0 \\
		0 &                0 &                0 & 0 & 0 &  \frac{S_0}{S_3+S_0}  
	\end{array}\right].
	$$
	
	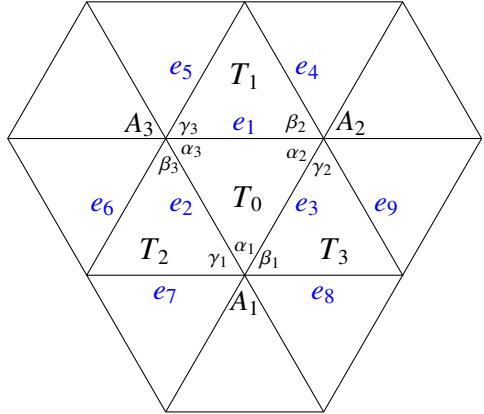
\begin{figure}[htbp]
		\centering
		\begin{tikzpicture}[scale=0.7]
			\path coordinate (A) at (0,0)
			coordinate (B) at (60:3)
			coordinate (C) at (120:3)
			coordinate (E) at (180:3)
			coordinate (F) at (0:3)
			coordinate (D) at ($ (C)!-1!(E) $);
			
			\draw[line width=.4pt]  (A) -- (B) -- (C) -- cycle;
			\draw[line width=.4pt]  (A) -- (F) -- (B) -- (D) -- (C) -- (E)-- cycle;
			
			\path coordinate (A1) at ($ (B)!1.5cm!5:(C) $)
			coordinate (B1) at ($ (A)!.5!(C) $)
			coordinate (C1) at ($ (A)!.5!(B) $)
			coordinate (A2) at ($ (A1)!.4!(C) $)
			coordinate (B2) at ($ (B1)!.4!(A) $)
			coordinate (C2) at ($ (C1)!.4!(B) $)
			coordinate (A3) at ($ (A1)!0.6cm!90:(B) $)
			coordinate (B3) at ($ (B1)!0.6cm!90:(C) $)
			coordinate (C3) at ($ (C1)!0.6cm!90:(A) $)
			coordinate (O) at ($ (A1)!0.56cm!-90:(B) $);
			
			\path   coordinate (C4) at ($ (C)!0.25cm!-10:(A) $)
			coordinate (B4) at ($ (B)!0.5cm!-2:(C) $)
			coordinate (A4) at ($ (A)!0.67cm!-10:(B) $)
			coordinate (D1) at ($ (D)!.35!(A1) $)
			coordinate (E1) at ($ (E)!.35!(B1) $)
			coordinate (F1) at ($ (F)!.35!(C1) $);
			
			\path coordinate (A5) at ($ (A)!0.5cm!-10:(F) $)
			coordinate (A6) at ($ (A)!0.5cm!8:(E) $)
			coordinate (B5) at ($ (B)!0.3cm!-10:(D) $)
			coordinate (B6) at ($ (B)!0.7cm!3:(F) $)
			coordinate (C5) at ($ (C)!0.6cm!-2:(E) $)
			coordinate (C6) at ($ (C)!0.7cm!-10:(D) $);
			
			\path coordinate (C7) at ($ (B)!.5!(D) $)
			coordinate (B7) at ($ (C)!.5!(D) $)
			coordinate (A7) at ($ (C)!.5!(E) $)
			coordinate (C8) at ($ (A)!.5!(E) $)
			coordinate (B8) at ($ (A)!.5!(F) $)
			coordinate (A8) at ($ (B)!.5!(F) $);
			
			\path coordinate (G) at ($(A)!-1!(B)$)
			coordinate (H) at ($(A)!-1!(C)$)
			coordinate (I) at ($(B)!-1!(C)$)
			coordinate (J) at ($(B)!-1!(A)$)
			coordinate (K) at ($(C)!-1!(A)$)
			coordinate (L) at ($(C)!-1!(B)$);
			
			\draw[line width=.4pt]  (A) -- (G)   (A) -- (H)   (B) -- (I)  (B) -- (J)  (C) -- (K)  (C) -- (L);
			\draw[line width=.4pt]  (G) -- (H) -- (F) -- (I) -- (J) -- (D)  -- (K) -- (L)-- (E) -- cycle;
			
			\node [below] at (O) {$T_0$};
			\node [below] at (D1) {$T_1$};
			\node [right] at (E1) {$T_2$};
			\node [left] at (F1) {$T_3$};
			
			\path coordinate (AA) at ($(A)!0.1cm!-145:(B)$)
			coordinate (BB) at ($(B)!0.9cm!174:(A)$)
			coordinate (CC) at ($(C)!0.85cm!185:(A)$);
			\node [below,font=\small] at (AA) {$A_1$};
			\node [below,font=\small] at (BB) {$A_2$};
			\node [below,font=\small] at (CC) {$A_3$};
			
			\node[above,blue,font=\small] at (A1) {$e_1$};
			\node[left,blue,font=\small] at (B1) {$e_2$};
			\node[right,blue,font=\small] at (C1) {$e_3$};
			
			\node[right,font=\tiny] at (C4) {$\alpha_3$};
			\node[below,font=\tiny] at (B4) {$\alpha_2$};
			\node[left,font=\tiny] at (A4) {$\alpha_1$};
			
			\node[above,font=\tiny] at (A5) {$\beta_1$};
			\node[above,font=\tiny] at (A6) {$\gamma_1$};
			\node[left,font=\tiny] at (B5) {$\beta_2$};
			\node[left,font=\tiny] at (B6) {$\gamma_2$};
			\node[right,font=\tiny] at (C5) {$\beta_3$};
			\node[below,font=\tiny] at (C6) {$\gamma_3$};
			
			\node[right,blue,font=\small] at (C7) {$e_4$};
			\node[below,blue,font=\small] at (C8) {$e_7$};
			\node[left,blue,font=\small] at (B7) {$e_5$};
			\node[below,blue,font=\small] at (B8) {$e_8$};
			\node[left,blue,font=\small] at (A7) {$e_6$};
			\node[right,blue,font=\small] at (A8) {$e_9$};
		\end{tikzpicture}
		\caption{Illustration of all kernel basis functions upon one cell}\label{fig:Tpatch_grid}
	\end{figure}
	As  $\displaystyle\det(\mathbf{A})=\frac13 \prod_{i=1:3} \frac{S_0}{S_0+S_i}$ and $\left\{
	\uw{}_{T_0,e_2,e_3},
	\uw{}_{T_0,e_3,e_1},
	\uw{}_{T_0,e_1,e_2},
	\uw{}_{T_0,e_1},
	\uw{}_{T_0,e_2},
	\uw{}_{T_0,e_3}\right\}$ are linearly independent, 
	$\left\{
	\upsi^{A_2}|_{T_0},
	\upsi^{A_3}|_{T_0},
	\upsi_{T_0}|_{T_0},
	\upsi_{T_1}|_{T_0},
	\upsi_{T_2}|_{T_0},
	\upsi_{T_3}|_{T_0}
	\right\}$ are linearly independent. 
\end{proof}

\begin{remark}\label{rem:cell-indep}
	If a cell $T_0$ has one (or more) vertex aligned on the boundary, then it will be covered by no more than two interior vertex patches and be contained in supports of no more than six vertex- or cell-related kernel basis functions; the restriction of these six functions on $T_0$ are linearly independent. 
\end{remark}

\begin{lemma}\label{lemma:kernel basis}
	The functions of $\Phi_h(\mathcal{T}_h):=\{\upsi^A,\ A\in\mathcal{X}_h^i;\ \upsi_T,\ T\in\mathcal{T}_h^i\}$ form a basis of $\uZ_{h0}$. 
\end{lemma}
\begin{proof}
	We only have to prove the functions of $\Phi_h(\mathcal{T}_h)$ are linearly independent. Indeed, provided that the set $\Phi_h(\mathcal{T}_h)$ is linearly independent, $\dim({\rm span}(\Phi_h(\mathcal{T}_h))) = \# \mathcal{X}^i_h + \# \mathcal{T}^i_h = 3 \# \mathcal{X}^i_h -2 = 3\# \mathcal{E}^i_h - (3 \# \mathcal{T}_h - 1) =\dim (\uV^{\rm sBDFM}_{h0})-\dim(\mathbb{P}^1_{h0})=\dim(\uV^{\rm sBDFM}_{h0})-\dim(\dv\,\uV^{\rm sBDFM}_{h0})=\dim(\uZ_{h0})$, and thus $\uZ_{h0}={\rm span}\left(\Phi_h(\mathcal{T}_h)\right)$.
	
	Now, given $\displaystyle\upsi_h=\sum_{A\in\mathcal{X}_h^i}c_A\upsi^A+\sum_{T\in\mathcal{T}_h^i}c_T\upsi_T=0$, we are going to show all $c_A$ and $c_T$ are zero. Similar to \cite{Zhang.S2020ima}, we adopt a sweeping process here. Given $a\in\mathcal{X}_h^b$, let $T$ be such that $a$ is a vertex of $T$. Then
	$$
	\upsi_h|_T=\sum_{A\in\mathcal{X}_h^i\cap \overline{T}}c_A\upsi^A|_T+\sum_{T'\in\mathcal{T}_h^i,T'\ \mbox{and}\ T\ \mbox{share\ a\ common\ edge}}c_{T'}\upsi_{T'}|_T=0.
	$$
	By Lemma \ref{independence on boundary-cell} and Remark \ref{rem:cell-indep}, $c_A=0$ for $A\in \mathcal{X}_h^i\cap\overline{T}$ and $c_{T'}=0$ for $T'\in\mathcal{T}_h^i$, such that $T'$ and $T$ share a common edge. Therefore, $c_A=0$ for any vertex $A\in\mathcal{X}_h^i$ that is connected to one boundary vertex $a\in\mathcal{X}_h^b$, and $c_T=0$ for any $T\in\mathcal{T}_h^i$ that connects to a boundary vertex $a\in\mathcal{X}^b_h$. 
	Similarly, we can show 
	$$
	c_A=0\ \forall\,A\in\mathcal{X}_h^{b,+2},\ \ \ c_T=0\ \forall\,T\in\mathcal{T}_h\ \mbox{that\ connects\ to } \mathcal{X}_h^{b,+1}.
	$$
	Repeating the procedure recursively, finally, we obtain
	$$
	c_A=0 \ \forall\, A \in \mathcal{X}_h^{b,+k},\ \ \ 
	c_T=0\ \forall\,T\in\mathcal{T}_h\ \mbox{that\ connects\ to } \mathcal{X}_h^{b,+(k-1)}
	$$
	where $k$ is the number of levels of the triangulation $\mathcal{T}_h$.
	Therefore, $c_A$ and $c_T$ are all zero and the functions of $\Phi_h(\mathcal{T}_h)$ are linearly independent. 
	The proof is completed. 
\end{proof}

\subsubsection{Basis functions of $V^{2+}_{h0}$}
Note that $\curl$ is a bijection from $V^{2+}_{h0}$ onto $\uZ_{h0}$. Therefore, the basis functions of $V^{2+}_{h0}$ are $\{\zeta^A,\ A\in\mathcal{X}_h^i;\ \zeta_T,\ T\in\mathcal{T}_h^i\}$, such that $\curl\,\zeta^A=\upsi^A$ and $\curl\, \zeta_T=\upsi_T$. More precisely(cf. Figure \ref{fig:basissupp}), 
\begin{equation}\label{potential_psiO}
	\zeta^O =
	\left\{
	\begin{aligned}
		& \lambda_0^2
		+ \frac{d_{m+1} d_{m+2} \sin{(\alpha_1+\beta_1)}}{2(S_1+S_2)} \lambda_0 \lambda_1(3\lambda_m-1)
		+ \frac{d_{m+1} d_{2m} \sin{(\alpha_m+\beta_m)}}{2(S_1+S_m)} \lambda_0 \lambda_m(3\lambda_1-1),\ in \ T_1,\\
		& \lambda_0^2
		+ \frac{d_{m+i} d_{m+i+1} \sin{(\alpha_i+\beta_i)}}{2(S_i+S_{i+1})} \lambda_0 \lambda_i(3\lambda_{i-1}-1)
		+ \frac{d_{m+i} d_{m+i-1} \sin{(\alpha_{i-1}+\beta_{i-1})}}{2(S_i+S_{i-1})} \lambda_0 \lambda_{i-1}(3\lambda_i-1),\ in \ T_i,\\
		&\qquad \qquad \qquad \qquad \qquad \qquad \qquad \qquad \qquad \qquad \qquad \qquad \qquad \qquad \qquad   (i=2:m-1)\\
		& \lambda_0^2
		+ \frac{d_{2m} d_{m+1} \sin{(\alpha_m+\beta_m)}}{2(S_m+S_1)} \lambda_0 \lambda_m(3\lambda_{m-1}-1)
		+ \frac{d_{2m} d_{2m-1} \sin{(\alpha_{m-1}+\beta_{m-1})}}{2(S_m+S_{m-1})} \lambda_0 \lambda_{m-1}(3\lambda_m-1),\ in \ T_m,\\
	\end{aligned}
	\right.
\end{equation}
and
\begin{equation}\label{potential_psiO}
	\zeta_{T_0}=
	\left\{
	\begin{aligned}
		& \frac{S_1}{S_1+S_0} \lambda_2\lambda_3(3\lambda_4-1),\ in \ T_1,\\
		& \frac{S_2}{S_1+S_0} \lambda_1\lambda_3(3\lambda_5-1),\ in \ T_2,\\
		& \frac{S_3}{S_1+S_0} \lambda_1\lambda_2(3\lambda_6-1),\ in \ T_3,\\
		& \frac{S_1}{S_1+S_0} \lambda_2\lambda_3(3\lambda_1-1) 
		+ \frac{S_2}{S_1+S_0} \lambda_1\lambda_3(3\lambda_2-1)
		+ \frac{S_3}{S_1+S_0} \lambda_1\lambda_2(3\lambda_3-1)
		- 6 \lambda_1\lambda_2\lambda_3,\ in \ T_0.
	\end{aligned}
	\right.
\end{equation}

\section{An enriched linear -- constant finite element scheme for incompressible flows}
\label{sec:elpair}

\subsection{An enriched linear element space}

Define
$$
\uV^{\rm el}_h:=\{\uv_h\in H(\dv,\Omega):\uv{}_h|_T\in \uP^{1+}(T),\ \int_e\uv{}_h\cdot\mathbf{t}\ \mbox{is\ continuous\ across\ interior\ edge}\ e\},
$$
and
$$
\uV{}^{\rm el}_{h0}:=\{\uv_h\in \uV^{\rm el}_h\cap H_0(\dv,\Omega): \int_e\uv{}_h\cdot\mathbf{t}\ \mbox{vanishes\ on\ boundary\ edge}\ e\}.
$$
\begin{remark}\label{rem:velcontained}
	Evidently, $\uV^{\rm el}_h=\{\uv{}_h\in \uV^{\rm sBDFM}_h:\dv\,\uv{}_h\in\mathbb{P}^0_{h0}\}$, and $\uV^{\rm el}_{h0}=\{\uv{}_h\in \uV^{\rm sBDFM}_{h0}:\dv\,\uv{}_h\in\mathbb{P}^0_{h0}\}$. Particularly, $\{\uv_h\in \uV^{\rm el}_{h0}:\dv\,\uv_h=0\}=\{\uv_h\in \uV^{\rm sBDFM}_{h0}:\dv\,\uv_h=0\}$.
\end{remark}

\begin{lemma}\label{lem:essBDFM}
	The exact sequence holds as
	\begin{equation}
		\{0\} \rightarrow V^{2+}_{h0} \xrightarrow{{\curl}} \uV{}^{\rm el}_{h0} \xrightarrow{{\dv}} \mathbb{P}^0_{h0}  \xrightarrow{\int_\Omega\cdot} \{0\}.
	\end{equation}
\end{lemma}

\begin{lemma}
	It can be verified that $\uV^{\rm el}_{h0}=\uV^{\rm ZZZ}_{h0}\cap \uV^{\rm MTW}_{h0}$.
\end{lemma}
\begin{proof}
	By definition, $\uV^{\rm el}_{h0}\subset \uV^{\rm sBDFM}_{h0}\subset \uV^{\rm ZZZ}_{h0}$ and $\uV^{\rm el}_{h0}\subset \uV^{\rm MTW}_{h0}$, namely $\uV^{\rm el}_{h0}\subset\uV^{\rm ZZZ}_{h0}\cap \uV^{\rm MTW}_{h0}$. On the other hand, given $\uv_h\in \uV^{\rm ZZZ}_{h0}\cap \uV^{\rm MTW}_{h0}$, $\uv_h|_T\in \uP_2(T)$, the normal component of $\uv_h|_T$ is piecewise linear, and $\dv\,\uv_h|_T$ is a constant on $T$ for any $T\in \mathcal{T}_h$; namely, $\uv_h|_T\in\uP^{1+}(T)$. Since all these three spaces $\uV^{\rm el}_{h0}$, $\uV^{\rm ZZZ}_{h0}$ and $\uV^{\rm MTW}_{h0}$ possess the same continuity, $\uV^{\rm el}_{h0}\supset\uV^{\rm ZZZ}_{h0}\cap \uV^{\rm MTW}_{h0}$.
\end{proof}

\subsubsection{Basis functions}
Firstly, we present associated with each edge $e\in\mathcal{E}_h^i$ a locally supported function $\upsi_e$. Given $e\in\mathcal{E}_h^i$, it may happen that both ends of $e$ are interior or that one end of $e$ is on the boundary; see Figure \ref{fig:edge-patch illu} below.
\begin{figure}[htbp]
	\centering
	\begin{tikzpicture}[scale=0.7]
		\path coordinate (A1) at (0,0)
		coordinate (A3) at (0:3)
		coordinate (A4) at (60:3)
		coordinate (A6) at (120:3)
		coordinate (A2) at (-60:3)
		coordinate (A5) at (-120:3);

		\node[left] at (A1) {$A_1$};
		\node[below] at (A5) {$A_5$};
		\node[below] at (A2) {$A_2$};
		\node[right] at (A3) {$A_3$};
		\node[right] at (A4) {$A_4$};
		\node[left] at (A6) {$A_6$};
		
		
		\draw[line width=.4pt]  (A1) -- (A2) ;
		\draw[line width=.4pt]  (A1) -- (A3) ;
		\draw[line width=.4pt]  (A1) -- (A4) ;
		\draw[line width=.4pt] (A5) -- (A2) -- (A3) -- (A4) -- (A6) -- (A1) -- cycle ;
		
		\path coordinate (E1) at ($(A2)!1.4cm!-3:(A1)$)
		coordinate (E) at ($(A3)!1.4cm!7:(A1)$)
		coordinate (E4) at ($(A1)!1.6cm!-2:(A4)$)
		coordinate (E2) at ($(A2)!.5!(A3)$)
		coordinate (E3) at ($(A3)!.5!(A4)$);
		
		\node [left,blue,font=\small] at (E1) {$ e_1 $};
		\node [above,blue,font=\small] at (E) {$ e $};
		\node [left,blue,font=\small] at (E4) {$ e_4 $};
		\node [right,blue,font=\small] at (E2) {$ e_2 $};
		\node [right,blue,font=\small] at (E3) {$ e_3 $};
		
		\path coordinate (F1) at ($(A1)!0.22cm!55:(A3)$)
		coordinate (F2) at ($(A2)!0.65cm!3:(A1)$)
		coordinate (F3) at ($(A3)!0.15cm!55:(A1)$);
		
		\node[right,font=\tiny] at (F2) {$ \alpha_2 $};
		\node[above left,font=\tiny] at (F3) {$ \alpha_3 $};
		\node[right,font=\tiny] at (F1) {$ \alpha_1 $};
		
		\path coordinate (B4) at ($(A4)!0.22cm!90:(A6)$)
		coordinate (B3) at ($(A3)!0.15cm!-6:(A1)$)
		coordinate (B1) at ($(A1)!0.35cm!-60:(A3)$);
		
		\node [below left,font=\tiny] at (B3) {$\beta_3$};
		\node [right,font=\tiny] at (B1) {$\beta_1$};
		\node [below,font=\tiny] at (B4) {$\beta_4$};
		
		\path coordinate (T3) at (-90:1.5)
		coordinate (T1) at (-40:1.5)
		coordinate (T2) at (40:1.5)
		coordinate (T4) at (95:1.5);
		
		\node[below] at (T3) {$ T_3 $};
		\node[right] at (T1) {$ T_1 $};
		\node[right] at (T2) {$ T_2 $};
		\node[above] at (T4) {$ T_4 $};		
	\end{tikzpicture}
	\hspace{0.25in}
	\begin{tikzpicture}[scale=0.7]
		\path coordinate (A1) at (0,0)
		coordinate (A3) at (0:3)
		coordinate (A4) at (60:3)
		coordinate (A6) at (120:3)
		coordinate (A2) at (-60:3)
		coordinate (A5) at (-120:3);

		\node[left] at (A1) {$A_1$};
		\node[left] at (A5) {$A_5$};
		\node[below] at (A2) {$A_2$};
		\node[right] at (A3) {$A_3$};
		\node[above] at (A4) {$A_4$};
		\node[left] at (A6) {$A_6$};
		
		
		\draw[line width=.4pt]  (A1) -- (A2) ;
		\draw[line width=.4pt]  (A1) -- (A3) ;
		\draw[line width=.4pt]  (A1) -- (A4) ;
		\draw[line width=.4pt] (A5) -- (A2) -- (A3) -- (A4) -- (A6) -- (A1) -- cycle ;
		
		\path coordinate (E1) at ($(A2)!1.4cm!-3:(A1)$)
		coordinate (E) at ($(A3)!1.4cm!7:(A1)$)
		coordinate (E4) at ($(A1)!1.6cm!-2:(A4)$)
		coordinate (E2) at ($(A2)!.5!(A3)$)
		coordinate (E3) at ($(A3)!.5!(A4)$);
		
		\node [left,blue,font=\small] at (E1) {$ e_1 $};
		\node [above,blue,font=\small] at (E) {$ e $};
		\node [left,blue,font=\small] at (E4) {$ e_4 $};
		\node [right,blue,font=\small] at (E2) {$ e_2 $};
		\node [right,blue,font=\small] at (E3) {$ e_3 $};
		
		\path coordinate (F1) at ($(A1)!0.22cm!55:(A3)$)
		coordinate (F2) at ($(A2)!0.65cm!3:(A1)$)
		coordinate (F3) at ($(A3)!0.15cm!55:(A1)$);
		
		\node[right,font=\tiny] at (F2) {$ \alpha_2 $};
		\node[above left,font=\tiny] at (F3) {$ \alpha_3 $};
		\node[right,font=\tiny] at (F1) {$ \alpha_1 $};
		
		\path coordinate (B4) at ($(A4)!0.22cm!90:(A6)$)
		coordinate (B3) at ($(A3)!0.15cm!-6:(A1)$)
		coordinate (B1) at ($(A1)!0.35cm!-60:(A3)$);
		
		\node [below left,font=\tiny] at (B3) {$\beta_3$};
		\node [right,font=\tiny] at (B1) {$\beta_1$};
		\node [below,font=\tiny] at (B4) {$\beta_4$};
		
		\path coordinate (T3) at (-90:1.5)
		coordinate (T1) at (-40:1.5)
		coordinate (T2) at (40:1.5)
		coordinate (T4) at (95:1.5);
		
		\node[below] at (T3) {$ T_3 $};
		\node[right] at (T1) {$ T_1 $};
		\node[right] at (T2) {$ T_2 $};
		\node[above] at (T4) {$ T_4 $};	
		
		\path coordinate (A8) at ($ (A3)!3cm!120:(A1) $)
		coordinate (A7) at ($ (A3)!3cm!-120:(A1) $);
		
		\draw [line width=.4pt] (A4) -- (A7) -- (A3) -- (A8) -- (A2) ;
		
		\node[right] at (A8) {$A_8$};
		
		\node[right] at (A7) {$A_7$};

		\path coordinate (T6) at ($(A3)!1.7cm!39:(A2)$)
		coordinate (T5) at ($(A3)!1.6cm!35:(A7)$);
		
		\node[below] at (T6) {$T_6$};
		\node[above right] at (T5) {$T_5$};
	\end{tikzpicture}
	\caption{Illustration of basis functions associated with interior edges}\label{fig:edge-patch illu}
\end{figure}
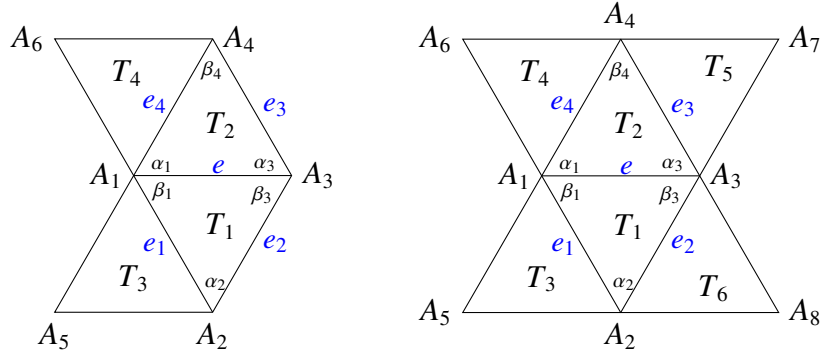

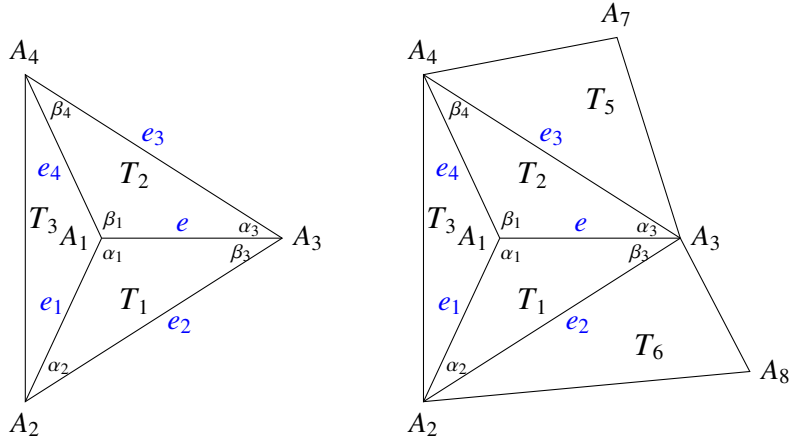
\begin{figure}[thbp]
	\centering
	\begin{tikzpicture}[scale=0.8]
		\path coordinate (A1) at (0,0)
		coordinate (A3) at (0:3)
		coordinate (A2) at (-115:3)
		coordinate (A4) at (115:3);
		
		\node [left,font=\small] at (A1) {$A_1$};
		\node [below,font=\small] at (A2) {$A_2$};
		\node [right,font=\small] at (A3) {$A_3$};
		\node [above,font=\small] at (A4) {$A_4$};

		\path coordinate (T1) at ($(A1)!0.85cm!-50:(A3)$)
		coordinate (T2) at ($(A1)!0.85cm!50:(A3)$)
		coordinate (T3) at ($(A1)!0.6cm!150:(A3)$);
		
		\node [below] at (T1) {$T_1$};
		\node [above] at (T2) {$T_2$};
		\node [left] at (T3) {$T_3$};
		
		\draw[line width=.4pt]  (A1) -- (A2) ;
		\draw[line width=.4pt]  (A1) -- (A3) ;
		\draw[line width=.4pt]  (A1) -- (A4) ;
		\draw[line width=.4pt]  (A2) -- (A3) -- (A4) -- cycle ;

		\path coordinate (E) at ($(A1)!.45!-3:(A3)$)
		coordinate (E1) at ($(A1)!.4!4:(A2)$)
		coordinate (E2) at ($(A2)!.5!-2:(A3)$)
		coordinate (E3) at ($(A3)!.5!(A4)$)
		coordinate (E4) at ($(A1)!.4!-2:(A4)$);

		\node [above,font=\small,blue] at (E) {$e$};
		\node [left,font=\small,blue] at (E1) {$e_1$};
		\node [right,font=\small,blue] at (E2) {$e_2$};
		\node [above,font=\small,blue] at (E3) {$e_3$};
		\node [left,font=\small,blue] at (E4) {$e_4$};
		
		\path coordinate (F1) at ($(A1)!0.2cm!0.5:(A3)$)
		coordinate (B1) at ($(A1)!0.22cm!-10:(A3)$)
		coordinate (F3) at ($(A1)!2.1cm!4:(A3)$)
		coordinate (B3) at ($(A1)!2cm!-7:(A3)$)
		coordinate (F2) at ($(A2)!0.6cm!6:(A1)$)
		coordinate (B4) at ($(A4)!0.65cm!-1:(A1)$);
		
		\node [below,font=\tiny] at (F1) {$\alpha_1$};
		\node [above,font=\tiny] at (B1) {$\beta_1$};
		\node [right,font=\tiny] at (F3) {$\alpha_3$};
		\node [right,font=\tiny] at (B3) {$\beta_3$};
		\node [right,font=\tiny] at (F2) {$\alpha_2$};
		\node [right,font=\tiny] at (B4) {$\beta_4$};
	\end{tikzpicture}
	\hspace{0.25in}
	\begin{tikzpicture}[scale=0.8]
		\path coordinate (A1) at (0,0)
		coordinate (A3) at (0:3)
		coordinate (A2) at (-115:3)
		coordinate (A4) at (115:3)
		coordinate (A7) at ($(A3)!3.5cm!-40:(A4)$)
		coordinate (A8) at ($(A3)!2.5cm!85:(A2)$);
		
		\node [left,font=\small] at (A1) {$A_1$};
		\node [below,font=\small] at (A2) {$A_2$};
		\node [right,font=\small] at (A3) {$A_3$};
		\node [above,font=\small] at (A4) {$A_4$};  	      
		\node [above,font=\small] at (A7) {$A_7$};
		\node [right,font=\small] at (A8) {$A_8$};
		
		\path coordinate (T1) at ($(A1)!0.85cm!-50:(A3)$)
		coordinate (T2) at ($(A1)!0.85cm!50:(A3)$)
		coordinate (T3) at ($(A1)!0.6cm!150:(A3)$)
		coordinate (T5) at ($(A3)!2.3cm!-55:(A1)$)
		coordinate (T6) at ($(A3)!1.5cm!70:(A1)$);
		
		\node [below] at (T1) {$T_1$};
		\node [above] at (T2) {$T_2$};
		\node [left] at (T3) {$T_3$};
		\node [above] at (T5) {$T_5$};
		\node [below] at (T6) {$T_6$};
		
		\draw[line width=.4pt]  (A1) -- (A2) ;
		\draw[line width=.4pt]  (A1) -- (A3) ;
		\draw[line width=.4pt]  (A1) -- (A4) ;
		\draw[line width=.4pt]  (A2) -- (A3) -- (A4) -- cycle ;
		\draw[line width=.4pt]  (A2) -- (A8) -- (A3) -- (A7) -- (A4) ;
		
		\path coordinate (E) at ($(A1)!.45!-3:(A3)$)
		coordinate (E1) at ($(A1)!.4!4:(A2)$)
		coordinate (E2) at ($(A2)!.5!-2:(A3)$)
		coordinate (E3) at ($(A3)!.5!(A4)$)
		coordinate (E4) at ($(A1)!.4!-2:(A4)$);

		\node [above,font=\small,blue] at (E) {$e$};
		\node [left,font=\small,blue] at (E1) {$e_1$};
		\node [right,font=\small,blue] at (E2) {$e_2$};
		\node [above,font=\small,blue] at (E3) {$e_3$};
		\node [left,font=\small,blue] at (E4) {$e_4$};
		
		\path coordinate (F1) at ($(A1)!0.2cm!0.5:(A3)$)
		coordinate (B1) at ($(A1)!0.22cm!-10:(A3)$)
		coordinate (F3) at ($(A1)!2.1cm!4:(A3)$)
		coordinate (B3) at ($(A1)!2cm!-7:(A3)$)
		coordinate (F2) at ($(A2)!0.6cm!6:(A1)$)
		coordinate (B4) at ($(A4)!0.65cm!-1:(A1)$);
		
		\node [below,font=\tiny] at (F1) {$\alpha_1$};
		\node [above,font=\tiny] at (B1) {$\beta_1$};
		\node [right,font=\tiny] at (F3) {$\alpha_3$};
		\node [right,font=\tiny] at (B3) {$\beta_3$};
		\node [right,font=\tiny] at (F2) {$\alpha_2$};
		\node [right,font=\tiny] at (B4) {$\beta_4$};
	\end{tikzpicture}
	\caption{Two cases of degeneration; see Remark \ref{rem:deg} below.}\label{fig:degeneration}
\end{figure}

If $e$ has a boundary vertex(e.g., in Figure \ref{fig:edge-patch illu}(left), $A_1\in \mathcal{X}_h^i, A_3 \in \mathcal{X}_h^b$), denote by
\begin{equation}
	\label{math expr:interior edge one boundary vertex}
	\upsi_e:=
	\left\{
	\begin{aligned}
		& \frac{S_3}{S_3+S_1} \uw_{T_3,e_1}, \ in \ T_3, \\
		& \uy_{T_1,e_1,e}
		+ \frac{d_1 \cos{\alpha_2}}{d_2}\uw_{T_1,e_1,e}
		+ \frac{S_3}{S_3+S_1}\uw_{T_1,e_1}
		+ \frac{\frac12 d_2 d_3 \sin{(\alpha_3 + \beta_3)} - (S_1+S_2)}{S_1+S_2} \uw_{T_1,e}, \ in \ T_1, \\
		& - \uy_{T_2,e_4,e}
		+ \frac{d_4 \cos{\alpha_4}}{d_3}\uw_{T_2,e_4,e}
		+ \frac{S_4}{S_4+S_2}\uw_{T_2,e_4}
		+ \frac{\frac12 d_2 d_3 \sin{(\alpha_3 + \beta_3)} - (S_1+S_2)}{S_1+S_2} \uw_{T_2,e}, \ in \ T_2, \\
		& \frac{S_4}{S_4+S_2} \uw_{T_4,e_4},\ in \ T_4.
	\end{aligned}
	\right.
\end{equation}

If both of the ends of $e$ are interior vertices(e.g., in Figure \ref{fig:edge-patch illu}(right), $A_1, A_3 \in \mathcal{X}_h^i$), denote by
\begin{equation}
	\label{math expr:interior edge zero boundary vertex}
	\upsi_e:=
	\left\{
	\begin{aligned}
		& \frac{S_3}{2(S_3+S_1)} \uw_{T_3,e_1}, \ in \ T_3, \\
		& \frac{S_4}{2(S_4+S_2)} \uw_{T_4,e_4},\ in \ T_4, \\
		& ( \frac{S_3}{2(S_3+S_1)} -1 ) \uw_{T_1,e_1}
		+ (1 - \frac{S_6}{2(S_6+S_1)} ) \uw_{T_1,e_2}
		+ \frac{\frac12 d_2 d_3 \sin{(\alpha_3 + \beta_3)} - \frac12 d_1 d_4 \sin{(\alpha_1 + \beta_1)} }{2(S_1+S_2)} \uw_{T_1,e} \\
		& \qquad \qquad \qquad \qquad \quad
		+ ( \frac{d_2 \cos{\beta_3}}{d} - \frac12 ) \uw_{T_1,e_1,e_2}
		+ \frac12 \uw_{T_1,e_1,e}
		- \frac12 \uw_{T_1,e_2,e}
		+ \uy_{T_1,e_1,e_2} , \ in \ T_1, \\
		& ( \frac{S_4}{2(S_4+S_2)} -1 ) \uw_{T_2,e_4}
		+ (1 - \frac{S_5}{2(S_5+S_2)} ) \uw_{T_2,e_3}
		+ \frac{\frac12 d_2 d_3 \sin{(\alpha_3 + \beta_3)} - \frac12 d_1 d_4 \sin{(\alpha_1 + \beta_1)} }{2(S_1+S_2)} \uw_{T_2,e} \\
		& \qquad \qquad \qquad \qquad \quad
		+ ( \frac12 - \frac{d_4 \cos{\beta_1}}{d} ) \uw_{T_2,e_3,e_4}
		+ \frac12 \uw_{T_2,e_4,e}
		- \frac12 \uw_{T_2,e_3,e}
		- \uy_{T_2,e_3,e_4} , \ in \ T_2, \\
		& - \frac{S_5}{2(S_5+S_2)} \uw_{T_5,e_3}, \ in \ T_5, \\
		& - \frac{S_6}{2(S_6+S_1)} \uw_{T_6,e_2}, \ in \ T_6.
	\end{aligned}
	\right.
\end{equation}

\begin{remark}\label{rem:deg}
	It is still possible that the support of a basis function associated with an interior edge could cover exactly three or five cells. They can be viewed as the degenerated cases, and the function $\upsi_e$ can be defined the same way. To be specific, when $T_3$ and $T_4$ coincide, the pattern in Figure \ref{fig:edge-patch illu}(left) would degenerate to a patch with three cells as shown in Figure \ref{fig:degeneration}(left); 
	moreover, $\upsi_{e}|_{T_3}=\frac{S_3}{S_3+S_1} \uw_{T_3,e_1} + \frac{S_3}{S_3+S_2} \uw_{T_3,e_4}$ 
	and $\upsi_{e}|_{T_i}(i=1,2)$ are same to their counterparts in (\ref{math expr:interior edge one boundary vertex}). 
	Correspondingly, the pattern in Figure \ref{fig:edge-patch illu}(right) would degenerate to a set of five cells as shown in Figure \ref{fig:degeneration}(right); 
	$\upsi_{e}|_{T_3}=\frac{S_3}{(2S_3+S_1)} \uw_{T_3,e_1} + \frac{S_3}{2(S_3+S_2)} \uw_{T_3,e_4}$ and $\upsi_{e}|_{T_i}(i=1,2,5,6)$  keep counterparts as (\ref{math expr:interior edge zero boundary vertex}).
\end{remark}

Now we are going to show all these $\{\upsi_e:\ e\in\mathcal{E}_h^i\}$ and $\{\upsi_T:\ T\in\mathcal{T}_h^i\}$ form a basis of $\uV^{\rm el}_{h0}$.

\begin{lemma}
	$\uV^{\rm el}_{h0}={\rm span}\{\upsi_e,\ e\in\mathcal{E}_h^i;\ \upsi_T,\ T\in\mathcal{T}_h^i\}$.
\end{lemma}
\begin{proof}
	Evidently, $\uV^{\rm el}_{h0}\supset{\rm span}\{\upsi_e,\ e\in\mathcal{E}_h^i;\ \upsi_T,\ T\in\mathcal{T}_h^i\}$. We turn to the other direction. 
	
	Firstly, we show ${\rm span}\{\dv\,\upsi_e,\ e\in\mathcal{E}_h^i\}=\mathbb{P}^0_{h0}.$ For both cases as in Figure \ref{fig:edge-patch illu}, $\dv\, \upsi_e=\frac{1}{S_1}$ on $T_1$ and $-\frac{1}{S_2}$ on $T_2$, and vanishes on all other cells. A simple algebraic argument leads to the assertion. 
	
	Secondly, all functions of $\uZ_{h0}$ can be represented by these functions. We only have to verify it for kernel functions each supported in a vertex patch. 
	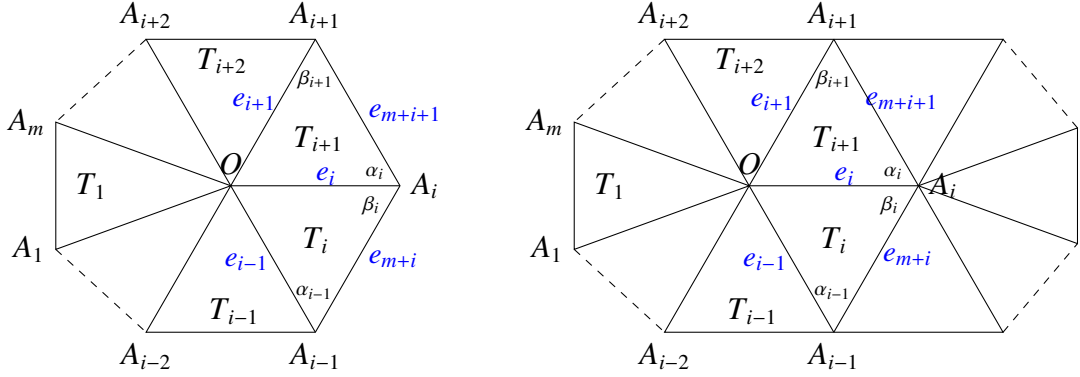
\begin{figure}[thbp]
		\centering
		\begin{tikzpicture}[scale=0.75]
			\path coordinate (O) at (0,0)
			coordinate (Ai) at (0:3)
			coordinate (Aip1) at (60:3)
			coordinate (Aip2) at (120:3)
			coordinate (Aim1) at (-60:3)
			coordinate (Aim2) at (-120:3)
			coordinate (Am) at (160:3.3)
			coordinate (A1) at (-160:3.3);
			
			\node [above] at (O) {$O$};
			\node [left] at (A1) {$A_1$};
			\node [below] at (Aim2) {$A_{i-2}$};
			\node [below] at (Aim1) {$A_{i-1}$};
			\node [right] at (Ai) {$A_i$};
			\node [above] at (Aip1) {$A_{i+1}$};
			\node [above] at (Aip2) {$A_{i+2}$};
			\node [left] at (Am) {$A_m$};
			
			\path coordinate (Eim1) at ($(O)!1.6cm!3:(Aim1)$)
			coordinate (Ei) at ($(O)!1.7cm!-6:(Ai)$)
			coordinate (Eip1) at ($(O)!1.8cm!-4:(Aip1)$)
			coordinate (Emi1) at ($(Ai)!.5!(Aip1)$)
			coordinate (Emi) at ($(Aim1)!.5!(Ai)$);
			
			\node [left,blue,font=\small] at (Eim1) {$e_{i-1}$};
			\node [above,blue,font=\small] at (Ei) {$e_i$};
			\node [left,blue,font=\small] at (Eip1) {$e_{i+1}$};
			\node [right,blue,font=\small] at (Emi) {$e_{m+i}$};
			\node [right,blue,font=\small] at (Emi1) {$e_{m+i+1}$};

			\path coordinate (T1) at ($(O)!2cm!-20:(A1)$)
			coordinate (Tim1) at ($(O)!1.8cm!-28:(Aim1)$)
			coordinate (Ti) at ($(O)!1.6cm!-20:(Ai)$)
			coordinate (Tip1) at ($(O)!2cm!-21:(Aip1)$)
			coordinate (Tip2) at ($(O)!1.8cm!-25:(Aip2)$);

			\node [left] at (T1) {$T_1$};
			\node [below] at (Tim1) {$T_{i-1}$};
			\node [below] at (Ti) {$T_i$}; 
			\node [below] at (Tip1) {$T_{i+1}$};
			\node [above] at (Tip2) {$T_{i+2}$};
			
			\path coordinate (Fi) at ($(Ai)!0.45cm!10:(O)$)
			coordinate (Bi) at ($(Ai)!0.5cm!-1:(O)$)
			coordinate (Fim1) at ($(Aim1)!0.4cm!-27:(O)$)
			coordinate (Bip1) at ($(Aip1)!0.35cm!32:(O)$);
			
			\node [above,font=\tiny] at (Fi) {$\alpha_i$};
			\node [below,font=\tiny] at (Bi) {$\beta_i$};
			\node [above,font=\tiny] at (Fim1) {$\alpha_{i-1}$};
			\node [below,font=\tiny] at (Bip1) {$\beta_{i+1}$};
			
			\draw[line width=.4pt] (O) -- (Am) -- (A1) -- cycle;
			\draw[line width=.4pt] (O) -- (Aim2) -- (Aim1) -- (Ai) -- (Aip1) -- (Aip2) -- cycle;
			\draw[line width=.4pt] (O) -- (Aim1) ;
			\draw[line width=.4pt] (O) -- (Ai) ;
			\draw[line width=.4pt] (O) -- (Aip1) ;
			
			\draw[dashed] (Aip2) -- (Am) ;
			\draw[dashed] (A1) -- (Aim2) ;
		\end{tikzpicture}
		\hspace{0.25in}
		\begin{tikzpicture}[scale=0.75]
			\path coordinate (O) at (0,0)
			coordinate (Ai) at (0:3)
			coordinate (Aip1) at (60:3)
			coordinate (Aip2) at (120:3)
			coordinate (Aim1) at (-60:3)
			coordinate (Aim2) at (-120:3)
			coordinate (Am) at (160:3.3)
			coordinate (A1) at (-160:3.3);
			
			\node [above] at (O) {$O$};
			\node [left] at (A1) {$A_1$};
			\node [below] at (Aim2) {$A_{i-2}$};
			\node [below] at (Aim1) {$A_{i-1}$};
			\node [right] at (Ai) {$A_i$};
			\node [above] at (Aip1) {$A_{i+1}$};
			\node [above] at (Aip2) {$A_{i+2}$};
			\node [left] at (Am) {$A_m$};
			
			\path coordinate (Eim1) at ($(O)!1.6cm!3:(Aim1)$)
			coordinate (Ei) at ($(O)!1.7cm!-6:(Ai)$)
			coordinate (Eip1) at ($(O)!1.8cm!-4:(Aip1)$)
			coordinate (Emi1) at ($(Ai)!.4!-16:(Aip1)$)
			coordinate (Emi) at ($(Aim1)!.5!3:(Ai)$);
			
			\node [left,blue,font=\small] at (Eim1) {$e_{i-1}$};
			\node [above,blue,font=\small] at (Ei) {$e_i$};
			\node [left,blue,font=\small] at (Eip1) {$e_{i+1}$};
			\node [right,blue,font=\small] at (Emi) {$e_{m+i}$};
			\node [above,blue,font=\small] at (Emi1) {$e_{m+i+1}$};

			\path coordinate (T1) at ($(O)!2cm!-20:(A1)$)
			coordinate (Tim1) at ($(O)!1.8cm!-28:(Aim1)$)
			coordinate (Ti) at ($(O)!1.6cm!-20:(Ai)$)
			coordinate (Tip1) at ($(O)!2cm!-21:(Aip1)$)
			coordinate (Tip2) at ($(O)!1.8cm!-25:(Aip2)$);

			\node [left] at (T1) {$T_1$};
			\node [below] at (Tim1) {$T_{i-1}$};
			\node [below] at (Ti) {$T_i$}; 
			\node [below] at (Tip1) {$T_{i+1}$};
			\node [above] at (Tip2) {$T_{i+2}$};
			
			\path coordinate (Fi) at ($(Ai)!0.45cm!10:(O)$)
			coordinate (Bi) at ($(Ai)!0.5cm!-1:(O)$)
			coordinate (Fim1) at ($(Aim1)!0.4cm!-27:(O)$)
			coordinate (Bip1) at ($(Aip1)!0.35cm!32:(O)$);
			
			\node [above,font=\tiny] at (Fi) {$\alpha_i$};
			\node [below,font=\tiny] at (Bi) {$\beta_i$};
			\node [above,font=\tiny] at (Fim1) {$\alpha_{i-1}$};
			\node [below,font=\tiny] at (Bip1) {$\beta_{i+1}$};
			
			\draw[line width=.4pt] (O) -- (Am) -- (A1) -- cycle;
			\draw[line width=.4pt] (O) -- (Aim2) -- (Aim1) -- (Ai) -- (Aip1) -- (Aip2) -- cycle;
			\draw[line width=.4pt] (O) -- (Aim1) ;
			\draw[line width=.4pt] (O) -- (Ai) ;
			\draw[line width=.4pt] (O) -- (Aip1) ;
			
			\draw[dashed] (Aip2) -- (Am) ;
			\draw[dashed] (A1) -- (Aim2) ;
			
			\path coordinate (D1) at ($(Ai)!3cm!60:(Aim1)$)
			coordinate (D4) at ($(Ai)!3cm!180:(Aim1)$)
			coordinate (D2) at ($(Ai)!3cm!100:(Aim1)$)
			coordinate (D3) at ($(Ai)!3cm!140:(Aim1)$);
			
			\draw[line width=.4pt] (Aim1) -- (D1) -- (Ai) ;
			\draw[line width=.4pt] (Ai) -- (D2) -- (D3) -- cycle;
			\draw[line width=.4pt] (Ai) -- (D4) -- (Aip1);
			\draw[dashed] (D1) -- (D2);
			\draw[dashed] (D3) -- (D4);
		\end{tikzpicture}
		\caption{Illustration of the interior edge $e_i$ with one(left) or two(right) interior vertices}\label{fig:edge with boundary or not}
	\end{figure}
	In fact, for an interior vertex $O$, $P_O=\cup_{i=1:m}T_i$,	$\overline{T}_{i}\cap \overline{T}_{i+1}=e_i$, $T_{m+1}=T_1$ and $e_i$ connects $O$ and $A_i$.	Denote for $i=1:m$ 
	$$
	\upsi_{e_i}^* =
	\left\{
	\begin{aligned}
		& \upsi_{e_i}, \ A_i \in \mathcal{X}_h^b, \\
		& \upsi_{e_i} + \frac12 \upsi_{T_i} + \frac12 \upsi_{T_{i+1}}, \ A_i \in \mathcal{X}_h^i.
	\end{aligned}
	\right.
	$$
	We refer to Figures \ref{fig:basissupp} and \ref{fig:edge-patch illu}, and formula \eqref{math expr:interior edge one boundary vertex}, \eqref{math expr:interior edge zero boundary vertex}, \eqref{math expr:Tpatch} and\eqref{math expr:vertex-centered basis} for the expressions of $\upsi^O$, $\upsi_{T_i}$ and $\upsi_{e_i}$. Then, ${\rm supp}(\upsi_{e_i}^*)=T_{i-1} \cup T_{i} \cup T_{i+1} \cup T_{i+2} \subset P_O$ in any event, and $\dv\,\sum_{i=1:m}\upsi_{e_i}^*=0$. Namely, $\sum_{i=1:m}\upsi_{e_i}^*\in \uZ_O={\rm span}\{\upsi^O\}$. A further calculation leads to $\sum_{i=1:m}\upsi_{e_i}^*=\upsi^O$, namely
	\begin{equation}\label{eq:repre}
		\displaystyle \upsi^{O}=
		\sum_{i=1:m} \upsi_{e_i}
		+ \frac{1}{2}\sum_{i=1:m, \ A_i \in \mathcal{X}_h^i}
		( \upsi_{T_i} +  \undertilde{\psi}_{T_{i+1}} ).
	\end{equation}
	
	Now, $\uV^{\rm el}_{h0}$ and ${\rm span}\{\upsi_e,\ e\in\mathcal{E}_h^i;\ \upsi_T,\ T\in\mathcal{T}_h^i\}$ have the same range under the operator $\dv$, and also $\uZ_{h0}\subset {\rm span}\{\upsi_e,\ e\in\mathcal{E}_h^i;\ \upsi_T,\ T\in\mathcal{T}_h^i\}$. Thus $\uV^{\rm el}_{h0}={\rm span}\{\upsi_e,\ e\in\mathcal{E}_h^i;\ \upsi_T,\ T\in\mathcal{T}_h^i\}$.  
	
	Further, $\dim({\rm span}\{\upsi_e,\ e\in\mathcal{E}_h^i;\ \upsi_T,\ T\in\mathcal{T}_h^i\})=\dim (\uV^{\rm el}_{h0})=\dim (\uZ_{h0}) + \dim (\mathbb{P}^0_{h0}) = \# \mathcal{X}^i_h + \# \mathcal{T}^i_h + \# \mathcal{T}_h - 1 = \# \mathcal{T}^i_h + \# \mathcal{E}^i_h=\#(\{\upsi_e,\ e\in\mathcal{E}_h^i;\ \upsi_T,\ T\in\mathcal{T}_h^i\})$. Therefore, the functions $\{\upsi_e,\ e\in\mathcal{E}_h^i;\ \upsi_T,\ T\in\mathcal{T}_h^i\}$ are linearly independent, and they form a set of basis of $\uV^{\rm el}_{h0}$. The proof is completed.
\end{proof}

\subsection{A lowest degree conservative scheme for the Stokes equation}
Denote 
$$
\mathbb{P}{}^0_h(\mathcal{T}_h):=\{q_h\in L^2(\Omega):q_h|_T\in P_0(T), \forall\, T \in \mathcal{T}_h\} \ \mbox{and}\ \ \mathbb{P}^0_{h0}(\mathcal{T}_h):=\mathbb{P}{}^0_h(\mathcal{T}_h)\cap L^2_0(\Omega).
$$
Based on the new finite element, a discretization scheme of~\eqref{eq:Stokes eq} is: Find $(\uu{}_{h},p_{h})\in \uV{}_{h0}^{\rm el}\times \mathbb{P}^0_{h0}$, such that
\begin{equation}\label{eq:discre Stokes eq}
	\left\{
	\begin{split}
		&\varepsilon^{2}\big(\nabla_{h}\,\uu{}_{h}, \nabla_{h}\,\uv{}_{h}\big)  -( \dv\,\uv{}_{h}, p_{h}) && = ( \uf,\uv{}_{h} ), & \forall\, \uv{}_{h}\in \uV{}_{h0}^{\rm el}, \\
		&(\dv\,\uu{}_{h}, q_{h} )&& = 0, & \forall\, q_{h}\in \mathbb{P}_{h0}^0.
	\end{split}
	\right.
\end{equation}
\begin{lemma}[Stability of $\uV^{\rm el}_{h0}-\mathbb{P}^0_{h0}$]\label{lem:stalep0} 
	It holds uniformly that
	\begin{equation}
		\inf_{q_h\in\mathbb{P}^0_{h0}}\sup_{\uv_h\in\uV^{\rm el}_{h0}}\frac{(\dv\,\uv{}_h,q_h)}{\|q_h\|_{0,\Omega}\|\uv_h\|_{1,h}}\geqslant C>0.
	\end{equation}
\end{lemma}
\begin{proof}
	Given $q_h\in \mathbb{P}^0_{h0}\subset\mathbb{P}^1_{h0}$, there exists $\uv{}_h\in \uV^{\rm sBDFM}_{h0}$, such that $\|\uv_h\|_{1,h}\leqslant C\|q_h\|_{0,\Omega}$ and $\dv\,\uv_h=q_h$, which implies $\uv_h\in \uV^{\rm el}_{h0}$. The proof is completed.
\end{proof}

\begin{lemma}\label{lem:approx}
	Given $\uw\in\uH^2(\Omega)$, it holds that
	\begin{equation}\label{eq:appvel}
		\inf_{\uv_h\in\uV^{\rm el}_h}\|\uw-\uv_h\|_{1,h}\leqslant Ch\|\uw\|_{2,\Omega}.
	\end{equation}
	
	Given $\uw\in\uH^2(\Omega)\cap \uH^1_0(\Omega)$ such that $\dv\,\uw=0$, it holds that
	\begin{equation}\label{eq:appvelker}
		\inf_{\uv_h\in \uV{}^{\rm sBDFM}_{h0},\,\dv\,\uv_h=0}\|\uw-\uv_h\|_{1,h}\leqslant Ch\|\uw\|_{2,\Omega}.
	\end{equation}
\end{lemma}
\begin{proof}
	Since linear element space is contained in $\uV^{\rm el}_{h0}$, \eqref{eq:appvel} holds directly. From Lemma \ref{lem:kernelappr} and Remark \ref{rem:velcontained}, \eqref{eq:appvelker} follows. The proof is completed. 
\end{proof}

The system \eqref{eq:discre Stokes eq} is uniformly well-posed by Brezzi's theory as below. 
\begin{lemma}
	The problem \eqref{eq:discre Stokes eq} admits a unique solution pair $(\uu_h,p_h)$, and
	\begin{equation}
		\varepsilon\|\uu{}_h\|_{1,h}+\frac{1}{\varepsilon}\|p_h\|_{0,\Omega}\cequiv \frac{1}{\varepsilon}\|\uf\|_{-1,h},
	\end{equation}
	where $\|\uf\|_{-1,h}:=\sup_{\uv_h\in \uV^{\rm el}_{h0}}\frac{(\uf,\uv_h)}{\|\uv_h\|_{1,h}}$.
\end{lemma}
\begin{proof}
	We only have to verify Brezzi's condition with respect to the parametrized norms.
\end{proof}

\begin{theorem}
	Let $(\uu,p)$ and $(\uu{}_{h},p_{h})$ be the solutions of~\eqref{eq:Stokes vf} and \eqref{eq:discre Stokes eq}, respectively. If $\uu\in \uH^2(\Omega)$ and $p\in H^1(\Omega)$, then
	\begin{equation}
		\|\uu-\uu_h\|_{1,h}\leqslant Ch\|\uu\|_{2,\Omega},\,\ \ \ \|p-p_h\|_{0,\Omega}\leqslant Ch(\varepsilon^2\|\uu\|_{2,\Omega}+\|p\|_{1,\Omega}).
	\end{equation}
	Here the constant $C$ does not depend on the parameter $\varepsilon$.
\end{theorem}
\begin{proof}
	The argument is quite standard, and we omit the details here. We only have to note that, since the scheme is strictly conservative, the solution of $\uu$ can be completely separated from $p$, and Lemma \ref{lem:approx} works here.
\end{proof}

\begin{remark}\label{rem:vp1}
	A further reduction of $\uV^{\rm el}_h$ leads to the spaces
	\begin{equation}
		\uV^1_{h}:=\{\uv_h\in H(\dv,\Omega):\ \uv_h|_T\in\uP_1(T),\ \forall\,T\in\mathcal{T},\ \int_e\uv_h\cdot\mathbf{t}\ \mbox{is\ continuous\ across}\ e\in\mathcal{E}_h^i\}
	\end{equation} 
	and 
	\begin{equation}
		\uV^1_{h0}:=\{\uv_h\in \uV^1_h\cap H_0(\dv,\Omega),\ \int_e\uv_h\cdot\mathbf{t}=0\ \mbox{on\ boundary\ edges}\ e\in\mathcal{E}_h^b\}.
	\end{equation}
	The pair $\uV^1_{h0}-\mathbb{P}^0_{h0}$ may be viewed as the most natural, if not the only, $\uP_1-P_0$ pair for the Stokes problem. Generally, this pair is not stable; we refer to Appendix \ref{sec:app} for a numerical verification. This way, we view the $\uV^{\rm el}_{h0}-\mathbb{P}^0_{h0}$ pair as a {\bf lowest-degree} stable conservative pair for the Stokes problem on general triangulations.
\end{remark}

\section{Numerical phenomena for eigenvalue problems}
\label{sec:num}

In this section, we test the numerical performance of the scheme for the Stokes eigenvalue problem: find $(\uu,p)\in\uH^1_0(\Omega)\times L^2_0(\Omega)$, such that 
\begin{equation}\label{eq:Stokesev}
	\left\{
	\begin{split}
		&\varepsilon^{2}\big(\nabla\,\uu, \nabla\,\uv\big)  -( \dv\,\uv, p) && = \lambda( \uu,\uv ), & \forall\, \uv\in \uH^1_0(\Omega), \\
		&(\dv\,\uu, q )&& = 0, & \forall\, q\in L^2_0(\Omega).
	\end{split}
	\right.
\end{equation}

Note that the two pairs ($\uV_{h0}^{\rm sBDFM}-\mathbb{P}^1_{h0}$ and $\uV_{h0}^{\rm el}-\mathbb{P}^0_{h0}$) lead to same computed eigenvalues on same grids. Series of numerical experiments are carried out and the computed eigenvalues are recorded below. For every example, we show the domain and initial grid in the left, and a list of computed values of the six lowest eigenvalues in the right. For these examples, we choose $\varepsilon=1$.

\paragraph{\bf Example 1}~

\makeatletter\def\@captype{figure}\makeatother
\begin{minipage}{0.25\textwidth}
	\centering
	\begin{tikzpicture}[scale=0.7]
		\path coordinate (O) at (0,0)
		coordinate (B) at (2,-2)
		coordinate (C) at (2,2)
		coordinate (A) at (-2,-2)
		coordinate (D) at (-2,2)
		coordinate (E) at (-1,1)
		coordinate (F) at (1,-1);
		
		\path coordinate (A1) at ($ (A)!.25!(B) $)
		coordinate (A2) at ($ (A)!.5!(B) $)
		coordinate (A3) at ($ (A)!.75!(B) $);

		\path coordinate (B1) at ($ (B)!.25!(C) $)
		coordinate (B2) at ($ (B)!.5!(C) $)
		coordinate (B3) at ($ (B)!.75!(C) $);
		
		\path coordinate (C1) at ($ (C)!.25!(D) $)
		coordinate (C2) at ($ (C)!.5!(D) $)
		coordinate (C3) at ($ (C)!.75!(D) $);
		
		\path coordinate (D1) at ($ (D)!.25!(A) $)
		coordinate (D2) at ($ (D)!.5!(A) $)
		coordinate (D3) at ($ (D)!.75!(A) $);

		\draw[line width=.4pt]  (A) -- (B) -- (C) -- (D) -- cycle;
		
		\draw[line width=.4pt]  (D1) -- (B3) ;
		\draw[line width=.4pt]  (D2) -- (B2) ;
		\draw[line width=.4pt]  (D3) -- (B1) ;
		
		\draw[line width=.4pt]  (A1) -- (C3) ;
		\draw[line width=.4pt]  (A2) -- (C2) ;
		\draw[line width=.4pt]  (A3) -- (C1) ;	
		
		\draw[line width=.4pt]  (D2) -- (C2) ;
		\draw[line width=.4pt]  (D3) -- (C1) ;
		\draw[line width=.4pt]  (A) -- (C) ;
		\draw[line width=.4pt]  (A1) -- (B3) ;
		\draw[line width=.4pt]  (A2) -- (B2) ;
		
		\draw[line width=.4pt]  (D) -- (E) ;	
		\draw[line width=.4pt]  (B) -- (F) ;
	\end{tikzpicture}
\end{minipage}
\makeatletter\def\@captype{table}\makeatother
\begin{minipage}{.7\textwidth}
	\centering
	\setlength{\tabcolsep}{1.5mm}{
		\begin{tabular}{ccccccc}
			\hline
			Mesh  & 0 & 1 & 2 & 3 & 4 & Trend \\
			\hline
			$\lambda_1$ & 66.4097  &  55.5965  &   53.1347  &  52.5407  & 52.3936  & $\searrow$  \\
			$\lambda_2$ & 123.5251 &  99.7536  &   94.0682  &  92.6136  & 92.2471  & $\searrow$  \\
			$\lambda_3$ & 137.3504 &  104.5997 &   95.1729  &  92.8802  & 92.3129  & $\searrow$  \\
			$\lambda_4$ & 165.0641 &  145.8915 &   132.8618 &  129.3819 & 128.5035 & $\searrow$  \\
			$\lambda_5$ & 201.2460 &  181.6767 &   161.1576 &  155.8845 & 154.5653 & $\searrow$  \\
			$\lambda_6$ & 203.7052 &  196.9708 &   174.6248 &  168.9307 & 167.5051 & $\searrow$  \\
			\hline
	\end{tabular}}
	\centering
\end{minipage}

\paragraph{\bf Example 2}~

\makeatletter\def\@captype{figure}\makeatother
\begin{minipage}{0.25\textwidth}
	\centering
	\begin{tikzpicture}[scale=0.7]
		\path coordinate (O) at (0,0)
		coordinate (B) at (270:2)
		coordinate (C) at (0:2)
		coordinate (E) at (90:2)
		coordinate (F) at (180:2)
		coordinate (A) at (-2,-2)
		coordinate (D) at (2,2);
		
		\path coordinate (A1) at ($ (A)!.125!(B) $)
		coordinate (A2) at ($ (A)!.25!(B) $)
		coordinate (A3) at ($ (A)!.375!(B) $)
		coordinate (A4) at ($ (A)!.5!(B) $)
		coordinate (A5) at ($ (A)!.625!(B) $)
		coordinate (A6) at ($ (A)!.75!(B) $)
		coordinate (A7) at ($ (A)!.875!(B) $);

		\path coordinate (B1) at ($ (B)!.125!(C) $)
		coordinate (B2) at ($ (B)!.25!(C) $)
		coordinate (B3) at ($ (B)!.375!(C) $)
		coordinate (B4) at ($ (B)!.5!(C) $)
		coordinate (B5) at ($ (B)!.625!(C) $)
		coordinate (B6) at ($ (B)!.75!(C) $)
		coordinate (B7) at ($ (B)!.875!(C) $);
		
		\path coordinate (C1) at ($ (C)!.125!(D) $)
		coordinate (C2) at ($ (C)!.25!(D) $)
		coordinate (C3) at ($ (C)!.375!(D) $)
		coordinate (C4) at ($ (C)!.5!(D) $)
		coordinate (C5) at ($ (C)!.625!(D) $)
		coordinate (C6) at ($ (C)!.75!(D) $)
		coordinate (C7) at ($ (C)!.875!(D) $);
		
		\path coordinate (D1) at ($ (D)!.125!(E) $)
		coordinate (D2) at ($ (D)!.25!(E) $)
		coordinate (D3) at ($ (D)!.375!(E) $)
		coordinate (D4) at ($ (D)!.5!(E) $)
		coordinate (D5) at ($ (D)!.625!(E) $)
		coordinate (D6) at ($ (D)!.75!(E) $)
		coordinate (D7) at ($ (D)!.875!(E) $);

		\path coordinate (E1) at ($ (E)!.125!(F) $)
		coordinate (E2) at ($ (E)!.25!(F) $)
		coordinate (E3) at ($ (E)!.375!(F) $)
		coordinate (E4) at ($ (E)!.5!(F) $)
		coordinate (E5) at ($ (E)!.625!(F) $)
		coordinate (E6) at ($ (E)!.75!(F) $)
		coordinate (E7) at ($ (E)!.875!(F) $);

		\path coordinate (F1) at ($ (F)!.125!(A) $)
		coordinate (F2) at ($ (F)!.25!(A) $)
		coordinate (F3) at ($ (F)!.375!(A) $)
		coordinate (F4) at ($ (F)!.5!(A) $)
		coordinate (F5) at ($ (F)!.625!(A) $)
		coordinate (F6) at ($ (F)!.75!(A) $)
		coordinate (F7) at ($ (F)!.875!(A) $);
		
		\draw[line width=.4pt]  (A) -- (B) -- (C) -- (D) -- (E) -- (F)-- cycle;
		\draw[line width=.4pt]  (A4) -- (E4) -- (C4) -- cycle;
		\draw[line width=.4pt]  (D4) -- (B4) -- (F4) -- cycle;
		
		\draw[line width=.4pt]  (C) -- (F) ;
		\draw[line width=.4pt]  (E) -- (B) ;
		\draw[line width=.4pt]  (A) -- (D) ;
	\end{tikzpicture}
\end{minipage}
\makeatletter\def\@captype{table}\makeatother
\begin{minipage}{.7\textwidth}
	\centering
	\setlength{\tabcolsep}{1.5mm}{
		\begin{tabular}{ccccccc}
			\hline
			Mesh  & 0 & 1 & 2 & 3 & 4  & Trend \\
			\hline
			$\lambda_1$  & 86.6443  & 83.3799   &  81.4757   &  80.9330   & 80.7931  & $\searrow$  \\
			$\lambda_2$  & 137.7299 & 113.2535  &  105.8261  &  103.8102  & 103.2968 & $\searrow$  \\
			$\lambda_3$  & 186.2746 & 177.2660  &  157.0575  &  151.3276  & 149.9012 & $\searrow$  \\
			$\lambda_4$  & 219.7048 & 179.7712  &  171.8289  &  170.2635  & 169.8594 & $\searrow$  \\
			$\lambda_5$  & 225.8015 & 216.8896  &  204.0614  &  199.9510  & 198.8708 & $\searrow$  \\
			$\lambda_6$  & 247.3904 & 269.6167  &  223.8862  &  211.9163  & 208.9613 & $\searrow$  \\
			\hline
	\end{tabular}}
\end{minipage}

\paragraph{\bf Example 3}~

\makeatletter\def\@captype{figure}\makeatother
\begin{minipage}{0.25\textwidth}
	\centering
	\includegraphics[width=4cm]{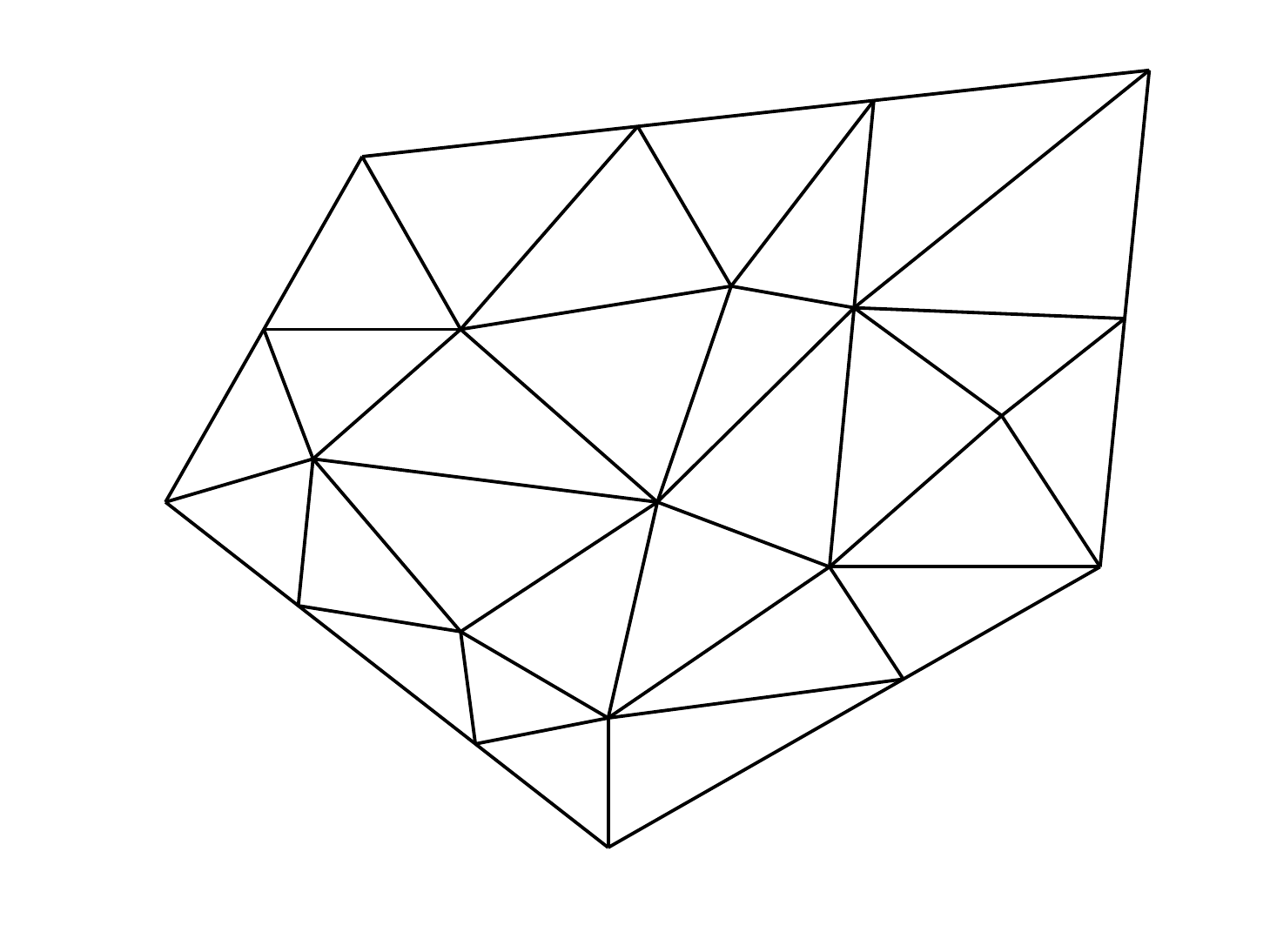}
\end{minipage}
\makeatletter\def\@captype{table}\makeatother
\begin{minipage}{.7\textwidth}
	\centering
	\setlength{\tabcolsep}{1.5mm}{
		\begin{tabular}{ccccccc}
			\hline
			Mesh  & 0 & 1 & 2 & 3 & 4 & Trend \\
			\hline
			$\lambda_1$ & 25.8121 &  23.3012 &  22.4095 &  22.1664 &  22.1039  & $\searrow$  \\
			$\lambda_2$ & 42.4798 &  37.8163 &  35.2351 &  34.5166 &  34.3322  & $\searrow$  \\
			$\lambda_3$ & 52.0032 &  46.2567 &  43.5630 &  42.8074 &  42.6114  & $\searrow$  \\
			$\lambda_4$ & 62.4579 &  61.7980 &  55.8809 &  54.1927 &  53.7558  & $\searrow$  \\
			$\lambda_5$ & 70.0038 &  66.4962 &  60.3462 &  58.6252 &  58.1810  & $\searrow$  \\
			$\lambda_6$ & 83.8312 &  82.6286 &  75.2565 &  72.9644 &  72.3525  & $\searrow$  \\
			\hline
	\end{tabular}}
\end{minipage}

\paragraph{\bf Example 4}~

\makeatletter\def\@captype{figure}\makeatother
\begin{minipage}{0.26\textwidth}
	\centering
	\includegraphics[width=4.3cm]{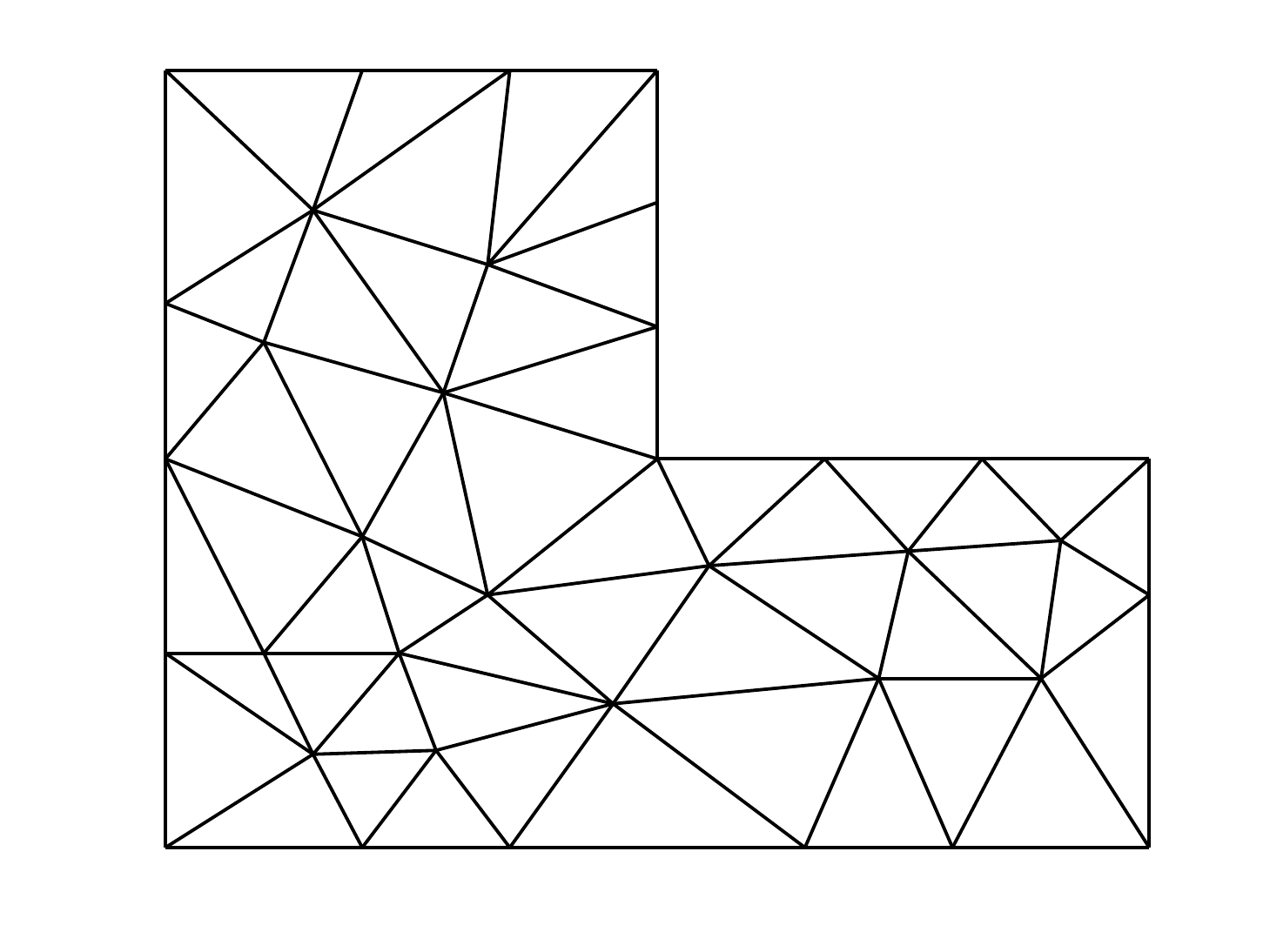}
\end{minipage}
\makeatletter\def\@captype{table}\makeatother
\begin{minipage}{.7\textwidth}
	\centering
	\setlength{\tabcolsep}{1.5mm}{
		\begin{tabular}{ccccccc}
			\hline
			Mesh  & 0 & 1 & 2 & 3 & 4 & Trend \\
			\hline
			$\lambda_1$ & 36.5520  &   33.5002  &   32.4349  &   32.1805  &   32.1302  & $\searrow$  \\
			$\lambda_2$ & 48.3991  &   39.8558  &   37.7611  &   37.2135  &   37.0697  & $\searrow$  \\
			$\lambda_3$ & 53.0517  &   45.0956  &   42.7302  &   42.1349  &   41.9870  & $\searrow$  \\
			$\lambda_4$ & 60.3236  &   53.6852  &   50.2220  &   49.2993  &   49.0633  & $\searrow$  \\
			$\lambda_5$ & 63.4514  &   60.6123  &   56.8284  &   55.7646  &   55.4969  & $\searrow$  \\
			$\lambda_6$ & 79.9487  &   76.4633  &   71.4214  &   69.9839  &   69.6178  & $\searrow$  \\
			\hline
	\end{tabular}}
\end{minipage}

\paragraph{\bf Example 5}~

\makeatletter\def\@captype{figure}\makeatother
\begin{minipage}{0.26\textwidth}
	\centering
	\includegraphics[width=4.4cm]{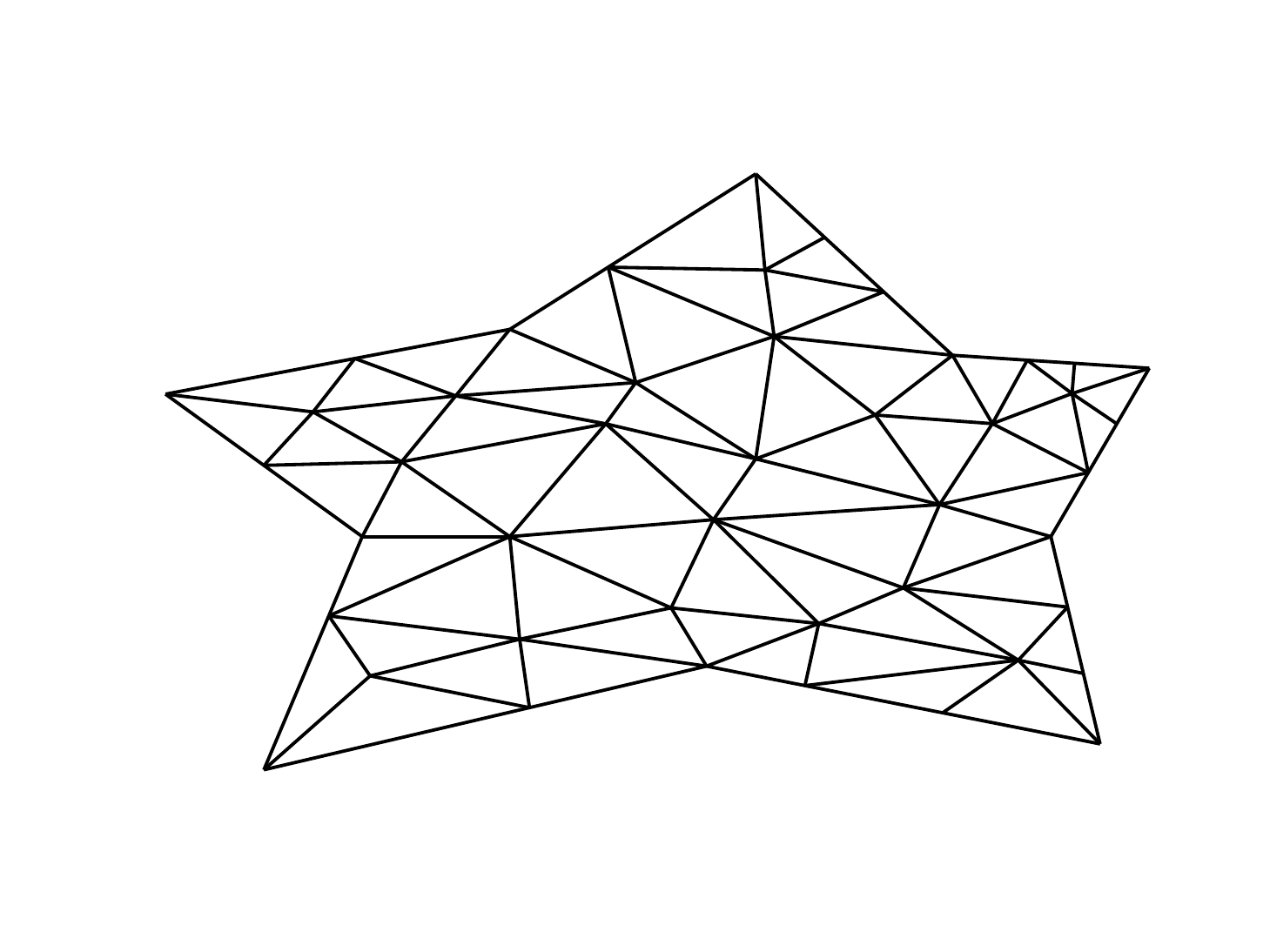}
\end{minipage}
\makeatletter\def\@captype{table}\makeatother
\begin{minipage}{.7\textwidth}
	\centering
	\setlength{\tabcolsep}{1.5mm}{
		\begin{tabular}{ccccccc}
			\hline
			Mesh  & 0 & 1 & 2 & 3 & 4 & Trend \\
			\hline
			$\lambda_1$ & 27.0359 &  25.1845 &  24.5809 &  24.4196  & 24.3798 & $\searrow$  \\
			$\lambda_2$ & 48.0933 &  44.7486 &  42.9104 &  42.4146  & 42.2914 & $\searrow$  \\
			$\lambda_3$ & 51.0107 &  44.8693 &  43.1253 &  42.6689  & 42.5528 & $\searrow$  \\
			$\lambda_4$ & 73.8299 &  64.1797 &  60.4715 &  59.4812  & 59.2282 & $\searrow$  \\
			$\lambda_5$ & 82.0255 &  69.5011 &  65.2793 &  64.1854  & 63.9103 & $\searrow$  \\
			$\lambda_6$ & 92.5390 &  84.0946 &  78.2179 &  76.6211  & 76.2141 & $\searrow$  \\
			\hline
	\end{tabular}}
\end{minipage}

It can be observed according to the experiments that
\begin{itemize}
	\item the computed eigenvalues converge to a limit in the speed of $\mathcal{O}(h^2)$;
	\item the computed eigenvalues decrease as the mesh is refined, which implies that the computed eigenvalues provide upper bounds of the exact eigenvalues. 
\end{itemize}

\section{Concluding remarks}
\label{sec:conc}

In this paper, a new conservative pair $\uV^{\rm el}_{h0}-\mathbb{P}^0_{h0}$ is established and shown stable for incompressible Stokes problem, and a numerical verification as in Appendix \ref{sec:app} illustrates that the $\uV^{\rm el}_{h0}-\mathbb{P}^0_{h0}$ pair is a lowest-degree one that is stable and conservative on general triangulations. The velocity component has an appearance of $H(\dv)$ element added with divergence-free bubble functions, and is comparable with ones given in, e.g., \cite{Madal.K;Tai.X;Winther.R2002,Xie;Xu;Xue2008,Guzman.J;Neilan.M2012}. However, the finite element space for velocity does not correspond to a Ciarlet's triple, and the construction and theoretical analysis can not be carried out in a usual way. The main technical ingredient is then to use an indirect approach by constructing and utilizing an auxiliary pair $\uV_{h0}^{\rm sBDFM}-\mathbb{P}^1_{h0}$.
~\\

The auxiliary pair $\uV_{h0}^{\rm sBDFM}-\mathbb{P}^1_{h0}$ is constructed by reducing $H(\dv)$ finite element spaces which was firstly adopted in \cite{Zeng.H;Zhang.C;Zhang.S2021}. It is interesting to notice that, the sBDFM element has the same nodal parameters as ones given in \cite{Madal.K;Tai.X;Winther.R2002,Xie;Xu;Xue2008} (the lowest-degree) and \cite{Guzman.J;Neilan.M2012} (the lowest-degree), but it uses the lowest-degree polynomials among these four, and only the sBDFM element space can accompany the piecewise linear polynomial space to form a stable pair, while the other three can only accompany the piecewise constant space. 
~\\

Besides, for conservative pairs in three-dimension, we refer to, e.g., \cite{Guzman;Neilan2018,Zhang3D2005,ZhangPS2011} where composite grids are required, as well as \cite{Guzman;Neilan2013} and \cite{Zhang3D2011} where high degree local polynomials are utilized. We refer to \cite{Chen;Dong;Qiao2013,Huang;Zhang2011,ZhangSY2009} for pairs on  rectangular grids and \cite{Neilan;Sap2016} for ones on cubic grids where full advantage of the geometric symmetry of the cells are taken. The approaches given in \cite{Zeng.H;Zhang.C;Zhang.S2021} and the present paper can be generalized to higher dimensions and non-simplicial grids. This will be discussed in future. 
~\\

Finally, it is worthy of noticing that, the finite element schemes given in the present paper, when used for the Stokes eigenvalue problem, can provide upper bounds for the exact eigenvalues. It has not been reported in the literature that nonconforming finite element schemes may provide upper bounds for the Stokes eigenvalue problem. In this paper, this unexpected phenomenon is illustrated by plenty of numerical experiments. Theoretical and further numerical investigation will be carried out in future.


%
%
%
\appendix

\section{A most natural linear--constant pair is not stable: a numerical verification}\label{sec:app}
In this section, we show by numerics the $\uV^1_{h0}-\mathbb{P}^0_{h0}$, defined in Remark \ref{rem:vp1}, is not stable on general triangulations, whereas
\begin{equation}\label{eq:degenerate}
	\displaystyle	\inf_{q_h\in \dv\,\uV^1_{h0}}\sup_{\uv_h\in \uV^1_{h0}}\frac{(\dv\,\uv_h,q_h)}{\|q_h\|_{0,\Omega}|\uv_h|_{1,h}}= \mathcal{O}(h)
\end{equation}
on a specific kind of triangulations. 
\subsection{A special triangulation and finite element space}
We consider the computational domain 
$\Omega= (0,1)\times(0,1) \setminus ( \{(x,y) : 0 \leqslant x \leqslant \frac12, x + \frac12 \leqslant y \leqslant 1\}
\cup \{ (x,y) : \frac12 \leqslant x \leqslant 1, 0 \leqslant y \leqslant x - \frac12 \})$. 
The initial triangulation is shown in Figure \ref{fig:gridP1P0}(left), and a sequence of triangulations are obtained by refining it uniformly(cf. Figure \ref{fig:gridP1P0}(right)).  
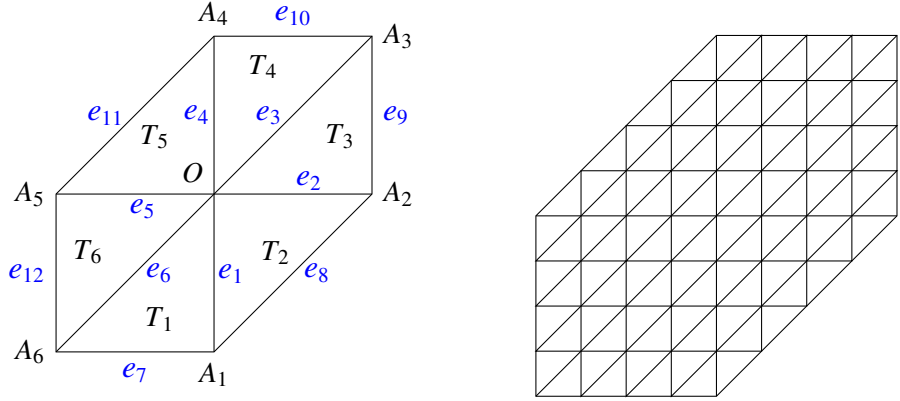
\begin{figure}[thbp]
	\centering
	\begin{tikzpicture}[scale=0.7]
		\path coordinate (O) at (0,0)
		coordinate (A1) at (-90:3)
		coordinate (A2) at (0:3)
		coordinate (A3) at (3,3)
		coordinate (A4) at (90:3)
		coordinate (A5) at (180:3)
		coordinate (A6) at (-3,-3);
		
		\node [above left,font=\small] at (O) {$O$};
		\node [below,font=\small] at (A1) {$A_1$};
		\node [right,font=\small] at (A2) {$A_2$};
		\node [right,font=\small] at (A3) {$A_3$};
		\node [above,font=\small] at (A4) {$A_4$};
		\node [left,font=\small] at (A5) {$A_5$};
		\node [left,font=\small] at (A6) {$A_6$};
		
		\draw[line width=.4pt]  (A1) -- (A4) ;
		\draw[line width=.4pt]  (A2) -- (A5) ;
		\draw[line width=.4pt]  (A3) -- (A6) ;
		\draw[line width=.4pt]  (A1) -- (A2) -- (A3) -- (A4) -- (A5) -- (A6) -- cycle ;
		
		\path coordinate (E1) at ($(A1)!.5!5:(O)$)
		coordinate (E2) at ($(A2)!.4!6:(O)$)
		coordinate (E3) at ($(A3)!.5!(O)$)
		coordinate (E4) at ($(A4)!.5!5:(O)$)
		coordinate (E5) at ($(A5)!.55!4:(O)$)
		coordinate (E6) at ($(A6)!.5!(O)$)
		coordinate (E7) at ($(A6)!.5!(A1)$)
		coordinate (E8) at ($(A1)!.5!(A2)$)
		coordinate (E9) at ($(A2)!.5!(A3)$)
		coordinate (E10) at ($(A3)!.5!(A4)$)
		coordinate (E11) at ($(A4)!.5!(A5)$)
		coordinate (E12) at ($(A5)!.5!(A6)$);
		
		\node [right,blue] at (E1) {$ e_1 $};
		\node [above,blue] at (E2) {$ e_2 $};
		\node [left,blue] at (E3) {$ e_3 $};
		\node [left,blue] at (E4) {$ e_4 $};
		\node [below,blue] at (E5) {$ e_5 $};
		\node [right,blue] at (E6) {$ e_6 $};
		\node [below,blue] at (E7) {$ e_7 $};
		\node [right,blue] at (E8) {$ e_8 $};
		\node [right,blue] at (E9) {$ e_9 $};
		\node [above,blue] at (E10) {$ e_{10} $};
		\node [left,blue] at (E11) {$ e_{11} $};
		\node [left,blue] at (E12) {$ e_{12} $};

		\path coordinate (T1) at ($(O)!2.2cm!-28:(A1)$)
		coordinate (T2) at ($(O)!1.35cm!-30:(A2)$)
		coordinate (T3) at ($(O)!2.2cm!-15:(A3)$)
		coordinate (T4) at ($(O)!2.2cm!-25:(A4)$)
		coordinate (T5) at ($(O)!1.3cm!-30:(A5)$)
		coordinate (T6) at ($(O)!2.2cm!-15:(A6)$);
		
		\node[below,font=\small] at (T1) {$ T_1 $};
		\node[below,font=\small] at (T2) {$ T_2 $};
		\node[right,font=\small] at (T3) {$ T_3 $};
		\node[above,font=\small] at (T4) {$ T_4 $};
		\node[above,font=\small] at (T5) {$ T_5 $};
		\node[left,font=\small] at (T6) {$ T_6 $};
	\end{tikzpicture}
	\hspace{0.5in}
	\begin{tikzpicture}[scale=0.6]
		\path coordinate (O) at (0,0)
		coordinate (B) at (270:4)
		coordinate (C) at (0:4)
		coordinate (E) at (90:4)
		coordinate (F) at (180:4)
		coordinate (A) at (-4,-4)
		coordinate (D) at (4,4);
		
		\path coordinate (A1) at ($ (A)!.125!(B) $)
		coordinate (A2) at ($ (A)!.25!(B) $)
		coordinate (A3) at ($ (A)!.375!(B) $)
		coordinate (A4) at ($ (A)!.5!(B) $)
		coordinate (A5) at ($ (A)!.625!(B) $)
		coordinate (A6) at ($ (A)!.75!(B) $)
		coordinate (A7) at ($ (A)!.875!(B) $);

		\path coordinate (B1) at ($ (B)!.125!(C) $)
		coordinate (B2) at ($ (B)!.25!(C) $)
		coordinate (B3) at ($ (B)!.375!(C) $)
		coordinate (B4) at ($ (B)!.5!(C) $)
		coordinate (B5) at ($ (B)!.625!(C) $)
		coordinate (B6) at ($ (B)!.75!(C) $)
		coordinate (B7) at ($ (B)!.875!(C) $);
		
		\path coordinate (C1) at ($ (C)!.125!(D) $)
		coordinate (C2) at ($ (C)!.25!(D) $)
		coordinate (C3) at ($ (C)!.375!(D) $)
		coordinate (C4) at ($ (C)!.5!(D) $)
		coordinate (C5) at ($ (C)!.625!(D) $)
		coordinate (C6) at ($ (C)!.75!(D) $)
		coordinate (C7) at ($ (C)!.875!(D) $);
		
		\path coordinate (D1) at ($ (D)!.125!(E) $)
		coordinate (D2) at ($ (D)!.25!(E) $)
		coordinate (D3) at ($ (D)!.375!(E) $)
		coordinate (D4) at ($ (D)!.5!(E) $)
		coordinate (D5) at ($ (D)!.625!(E) $)
		coordinate (D6) at ($ (D)!.75!(E) $)
		coordinate (D7) at ($ (D)!.875!(E) $);

		\path coordinate (E1) at ($ (E)!.125!(F) $)
		coordinate (E2) at ($ (E)!.25!(F) $)
		coordinate (E3) at ($ (E)!.375!(F) $)
		coordinate (E4) at ($ (E)!.5!(F) $)
		coordinate (E5) at ($ (E)!.625!(F) $)
		coordinate (E6) at ($ (E)!.75!(F) $)
		coordinate (E7) at ($ (E)!.875!(F) $);

		\path coordinate (F1) at ($ (F)!.125!(A) $)
		coordinate (F2) at ($ (F)!.25!(A) $)
		coordinate (F3) at ($ (F)!.375!(A) $)
		coordinate (F4) at ($ (F)!.5!(A) $)
		coordinate (F5) at ($ (F)!.625!(A) $)
		coordinate (F6) at ($ (F)!.75!(A) $)
		coordinate (F7) at ($ (F)!.875!(A) $);
		
		\draw[line width=.4pt]  (A) -- (B) -- (C) -- (D) -- (E) -- (F)-- cycle;
		\draw[line width=.4pt]  (A2) -- (E6) -- (C2) -- (A6) -- (E2) -- (C6) -- cycle;
		\draw[line width=.4pt]  (A4) -- (E4) -- (C4) -- cycle;
		
		\draw[line width=.4pt]  (D2) -- (B6) -- (F2) -- (D6) -- (B2) -- (F6) -- cycle;
		\draw[line width=.4pt]  (D4) -- (B4) -- (F4) -- cycle;
		
		\draw[line width=.4pt]  (C) -- (F) ;
		\draw[line width=.4pt]  (E) -- (B) ;
		\draw[line width=.4pt]  (A) -- (D) ;
	\end{tikzpicture}
	\caption{Left: the initial grid or a 6-cell patch. Right: the grid after twice refinement}\label{fig:gridP1P0}
\end{figure}

Given a patch $P_O$ as shown in Figure \ref{fig:gridP1P0}(left), denote by $ \uV_{h0}^1(P_O)= {\rm span} \{ \uphi_1^O, \uphi_2^O, \uphi_3^O \}$ and denote for $i=1:6, \ \uV_{h0}^1(T_i)= {\rm span} \{ \uphi_{T_i}^1, \uphi_{T_i}^2, \uphi_{T_i}^3 \}$. 
Specifically, 
$\uphi^O_s|_{T_i}=\uphi_{T_i}^s, s=1:2,i=1:6$ and
\begin{equation}
	\uphi^O_3 =
	\left\{
	\begin{aligned}
		& \uphi_{T_1}^1 - 2 \uphi_{T_1}^2 + \uphi_{T_1}^2 , \ in \ T_1;
		& \uphi_{T_2}^1 -  \uphi_{T_2}^2 - \uphi_{T_2}^2 , \ in \ T_2;\\
		& 2 \uphi_{T_3}^1 - \uphi_{T_3}^2 + \uphi_{T_3}^2 , \ in \ T_3;
		& \uphi_{T_4}^1 - 2 \uphi_{T_4}^2 + \uphi_{T_4}^2 , \ in \ T_4;\\
		& \uphi_{T_5}^1 - \uphi_{T_5}^2 - \uphi_{T_5}^2 , \ in \ T_5;
		& 2 \uphi_{T_6}^1 - \uphi_{T_6}^2 + \uphi_{T_6}^2 , \ in \ T_6;\\
	\end{aligned}
	\right.
\end{equation}
where for $i=1:6$, 
$
\uphi_{T_i}^1=
\left(
\begin{array}{c}
	\lambda_0 \\
	0 \\
\end{array}
\right)
$, 
$
\uphi_{T_i}^2=
\left(
\begin{array}{c}
	0 \\
	\lambda_0 \\
\end{array}
\right)
$,
and 
$
\uphi_{T_1}^3=
\left(
\begin{array}{c}
	\lambda_6-\lambda_1\\
	0
\end{array}
\right)
$, 
$
\uphi_{T_2}^3=
\left(
\begin{array}{c}
	\lambda_1-\lambda_2\\
	\lambda_1-\lambda_2
\end{array}
\right)
$, 
$
\uphi_{T_3}^3=
\left(
\begin{array}{c}
	0\\
	\lambda_2-\lambda_3
\end{array}
\right)
$, 
$
\uphi_{T_4}^3=
\left(
\begin{array}{c}
	\lambda_3-\lambda_4\\
	0
\end{array}
\right)
$, 
$
\uphi_{T_5}^3=
\left(
\begin{array}{c}
	\lambda_4-\lambda_5\\
	\lambda_4-\lambda_5
\end{array}
\right)
$, 
$
\uphi_{T_6}^3=
\left(
\begin{array}{c}
	0\\
	\lambda_5-\lambda_6
\end{array}
\right)
$.

Similar to Lemma \ref{lemma:kernel basis}, we can show the lemma below. 

\begin{lemma}
	$dim(\undertilde{V}_{h0}^{1})=3 \# \mathcal{X}_h^i$ 
	and 
	$\uV^1_{h0}={\rm span}\{ \uphi_1^A, \uphi_2^A, \uphi_3^A, A \in \mathcal{X}_h^i \}$.
\end{lemma}

\subsection{Numerical verification of the inf-sup constant}
By Courant's min-max theorem, it is easy to show the lemma below. 
\begin{lemma}
	With respect to any set of basis functions of $\uV^1_{h0}$ and $\mathbb{P}^0_h$, denote by $A$ the stiffness matrix of $(\nabla_h\,\cdot,\nabla_h\,\cdot)$ on $\uV^1_{h0}$, by $M$ the mass matrix on $\uV^1_{h0}$, and by $B$ the stiffness matrix of $(\dv\,\cdot,\cdot)$ on $\uV^1_{h0}\times \mathbb{P}^0_h$. Then
	$$
	\inf_{q_h\in \dv\,\uV^1_{h0}}\sup_{\uv_h\in \uV^1_{h0}}\frac{(\dv\,\uv_h,q_h)}{\|q_h\|_{0,\Omega}|\uv_h|_{1,h}}= \lambda^+_{\min},
	$$ 
	where $\lambda^+_{\min}$ is the smallest positive eigenvalue of the matrix eigenvalue problem $ BA^{-1}B^{T} {\rm v} = {\rm \lambda} M  {\rm v}$.
\end{lemma}

The maximum eigenvalue of the proposed eigenvalue problem is denoted by $\lambda_{\max}$. Table \ref{table:lambda_min} displays the computed values of $\lambda^+_{\min}$ and $\lambda_{\max}$ on a series of refined grids. And Figure \ref{fig:decay_loglog} illustrates that $\lambda^+_{\min}$ degenerates in the rate of $\mathcal{O}(h)$. This verifies \eqref{eq:degenerate} numerically. 

\makeatletter\def\@captype{table}\makeatother
\begin{minipage}{.4\textwidth}
	\centering
	\setlength{\tabcolsep}{2.5mm}{
		\begin{tabular}{cccc}
			\hline
			h & $\lambda_{\min}^+$ & Rate & $\lambda_{\max}$ \\
			\hline
			1/2  & 0.2232 & -      & 1.3822  \\
			1/4  & 0.1235 & 0.8538 & 1.4081  \\
			1/8  & 0.0636 & 0.9574 & 1.4131  \\
			1/16 & 0.0321 & 0.9865 & 1.4140  \\
			1/32 & 0.0161 & 0.9955 & 1.4142  \\
			1/64 & 0.0081 & 0.9911 & 1.4142  \\
			\hline
	\end{tabular}}
	\caption{Computed values of $\lambda_{\min}^+$ and $\lambda_{\max}$}\label{table:lambda_min}
\end{minipage}
\makeatletter\def\@captype{figure}\makeatother
\begin{minipage}{0.5\textwidth}
	\centering
	\includegraphics[width=6cm]{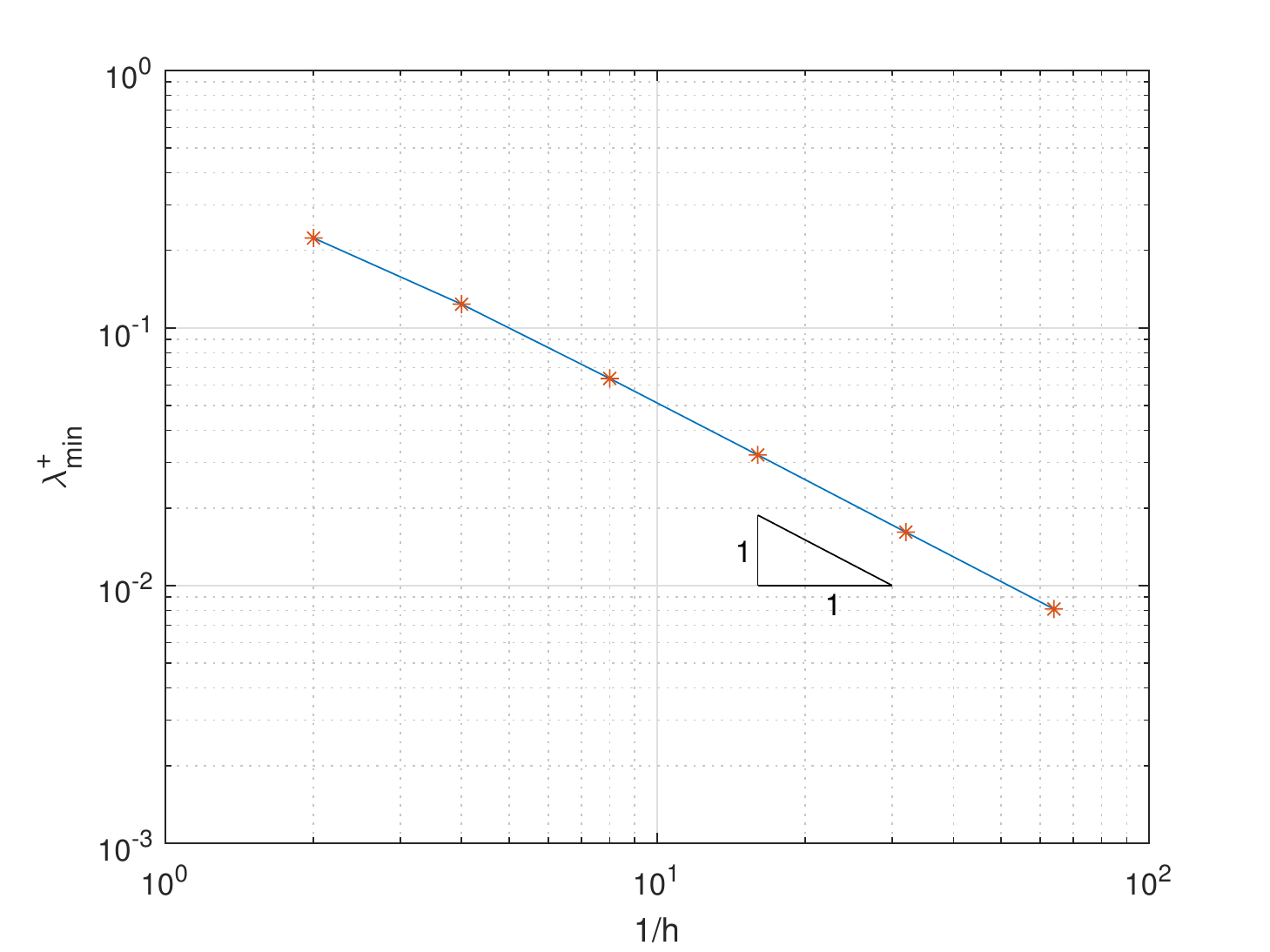}
	\caption{$\lambda_{min}^+$ decays along with mesh refinement}\label{fig:decay_loglog}
\end{minipage}
\bibliography{p1+p0}
\bibliographystyle{plain}
\end{document}